\documentclass[a4paper,12pt]{amsart}
\usepackage{amssymb,comment}
\usepackage{eepic}
\usepackage[dvipdfm]{graphicx}
\usepackage[all]{xy}

\usepackage[dvipdfm]{hyperref}
\makeatletter
\theoremstyle{plain}
  \newtheorem{thm}{Theorem}[section]
  \newtheorem{prop}[thm]{Proposition}
  \newtheorem{lem}[thm]{Lemma}
  \newtheorem{cor}[thm]{Corollary}
  
\theoremstyle{definition}
  \newtheorem{dfn}[thm]{Definition}
  \newtheorem{exmp}[thm]{Example}
  \newtheorem{ques}[thm]{Question}
  \newtheorem{rem}[thm]{Remark}

\numberwithin{equation}{section}

\let\opn\operatorname 
\newcommand\ba{\mathbf{a}}
\renewcommand\iff{\Longleftrightarrow}

\newcommand\NN{\mathbb{N}}
\newcommand\ZZ{\mathbb{Z}}
\newcommand\cA{\mathcal{A}}
\newcommand\cB{\mathcal{B}}
\newcommand\cF{\mathcal{F}}
\newcommand\cP{\mathcal{P}}
\newcommand\fb{\mathfrak b}

\newcommand\kk{\Bbbk}
\newcommand\wS{\widetilde{S}}
\newcommand\wI{\widetilde{I}}

\newcommand\wF{\widetilde{F}}

\newcommand\wQ{\widetilde{Q}}
\newcommand\wP{\widetilde{P}}
\newcommand\wg{\widetilde{g}}

\newcommand\BoX{\opn{\mathsf{b-pol}}}

\newcommand\chara{\opn{char}}
\newcommand\lcm{\opn{lcm}}

\newcommand\m{\mathsf{m}}
\newcommand\sq{\mathsf{sq}}
\newcommand\wm{\widetilde{\m}}
\newcommand\n{\mathsf{n}}
\newcommand\wn{\widetilde{\n}}

\newcommand\z{X}
\newcommand\too{\longrightarrow}
\newcommand\rmv{\opn{rm}}
\newcommand\ga{\alpha}
\newcommand\gl{\lambda}
\newcommand\gd{\delta}
\newcommand\gs{\sigma}
\newcommand\gt{\tau}

\newcommand\longto{\longrightarrow}
\newcommand\diff{\partial}
\let\wt\widetilde
\newcommand\void\varnothing
\def\<{\langle}
\def\>{\rangle}

\newcommand\bra[1]{\left[ #1\right]}

\let\bpol\BoX
\newcommand\bdeg[1]{\mathbf\deg}
\newcommand\Con{\opn{Con}}
\newcommand\embto{\hookrightarrow}
\newcommand\LCM{\opn{LCM}}
\title[Eliahou-Kervaire type resolution]{Alternative polarizations of Borel fixed ideals, Eliahou-Kervaire type resolution and \\
discrete Morse theory}
\author{Ryota Okazaki}
\thanks{The first author is partially supported by JST, CREST}
\address{Department of Pure and Applied Mathematics, 
Graduate School of Information Science and Technology,
Osaka University, Toyonaka, Osaka 
560-0043, Japan}
\email{r-okazaki@cr.math.osaka-u.ac.jp}
\author{Kohji Yanagawa}
\thanks{The second author is partially supported by Grant-in-Aid for Scientific Research (c) (no.22540057).}
\address{Department of Mathematics, Kansai University,
Suita 564-8680, Japan}
\email{yanagawa@ipcku.kansai-u.ac.jp}
\begin{document}
\maketitle
\begin{abstract}
We construct an Eliahou-Kervaire-like minimal free resolution of the alternative polarization $\BoX(I)$
of a Borel fixed ideal $I$. It yields new descriptions of the minimal free resolutions  
of $I$ itself and $I^\sq$, where $(-)^\sq$ is the squarefree operation in the shifting theory. 
These resolutions are cellular, 
and the (common) supporting cell complex is given by discrete Morse theory. 
If $I$ is generated in one degree, our description is equivalent to that of Nagel and Reiner. 
\end{abstract}
\section{Introduction}
Let $S:=\kk[x_1, \ldots, x_n]$ be a polynomial ring over a field $\kk$. 
For a monomial ideal $I \subset S$, $G(I)$ denotes the set of minimal (monomial) generators of $I$. 
We say a monomial ideal $I \subset S$ is {\it Borel fixed} (or {\it strongly stable}), if  
$\m \in G(I)$, $x_i|\m$ and $j <i$  imply $(x_j/x_i) \cdot \m \in I$.    
Borel fixed ideals are important, since they appear as the {\it generic initial ideals} 
of homogeneous ideals (if $\chara(\kk)=0$).  

As shown in \cite{AHH2}, a squarefree analog of a Borel fixed ideal is also important in combinatorial commutative algebra. 
We say a squarefree monomial ideal $I$ is {\it squarefree strongly stable},   
if  $\m \in G(I)$, $x_i|\m$, $x_j \! \not | \m$ and $j <i$  imply $(x_j/x_i) \cdot \m \in I$.
    
Any monomial $\m \in S$ with $\deg(\m) =e$ has a unique expression 
\begin{equation}\label{alpha expression}
\m = \prod_{i=1}^e x_{\alpha_i} \quad \text{with} \quad 1 \le \alpha_1 \le \alpha_2 \le \cdots 
\le \alpha_e \le n. 
\end{equation}
Now we can consider the squarefree monomial 
$$\m^\sq = \prod_{i=1}^e x_{\alpha_i+i-1}$$
in the ``larger" polynomial ring $T=\kk[x_1, \ldots, x_N]$ with $N \gg 0$. If $I \subset S$ is Borel fixed, then 
$$I^\sq:= ( \, \m^\sq \mid \m \in G(I)\, ) \subset T$$
is squarefree strongly stable. 
Moreover, for a Borel fixed ideal $I$ and all $i,j$, we have 
$\beta_{i,j}^S(I)=\beta_{i,j}^T(I^\sq)$.  See \cite{AHH2} for further information. 

A minimal free resolution of a Borel fixed ideal $I$ has been constructed   
by Eliahou and Kervaire \cite{EK}. 
While the minimal free resolution is unique up to isomorphism, 
its ``description" depends on the choice of a free basis, and further analysis of the minimal free resolution 
is still an interesting problem. See, for example,  \cite{BW,HT,JW,Mer,NR}.  
In this paper, we will give a new approach which is applicable to both $I$ and $I^\sq$. 
Our main tool is the ``alternative" polarization $\BoX(I)$ of $I$.

Let $\wS:=\kk[ \, x_{i,j} \mid 1 \le i \le n, \, 1 \le j \le d \, ]$ be the polynomial ring, and  set
$$\Theta := \{ x_{i,1}-x_{i,j} \mid 1 \leq i \leq n, \, 2 \leq j \leq d \, \} \subset \wS.$$ 
Then there is an isomorphism  $\wS/(\Theta) \cong S$ induced by  
$\wS \ni x_{i,j} \longmapsto x_i \in S$. 
Throughout this paper, $\wS$ and $\Theta$ are used in this meaning. 

Assume that $\m \in G(I)$ has the expression \eqref{alpha expression}. If $\deg (\m) \, (=e) \le d$, we set
\begin{equation}\label{b-pol}
\BoX(\m)= \prod_{i=1}^e x_{\alpha_i, i} \in \wS.
\end{equation} 
Note that $\BoX(\m)$ is a squarefree monomial.  
If there is no danger of confusion, $\BoX(\m)$ is denoted by $\wm$.  
If $\m=\prod_{i=1}^nx_i^{a_i}$, then we have 
$$\wm \ (=\BoX(\m)) \, =\prod_{\substack{1 \leq i \leq n \\b_{i-1}+1 \leq j \leq b_i }}x_{i,j} \in \wS, 
 \quad \text{where} \quad b_i:=\sum_{l=1}^ia_l.$$
If $\deg(\m) \le d$ for all $\m \in G(I)$, we set 
$$\BoX(I):= (\BoX(\m) \mid \m \in G(I) )\subset \wS.$$ 

Extending a result due to Nagel and Reiner (\cite{NR}), the second author showed that if $I$ is Borel fixed, then $\wI:=\BoX(I)$ is a polarization of $I$
(\cite[Theorem 3.4]{y10}), that is, 
$\Theta$ forms an $\wS/\wI$-regular sequence with the natural isomorphism $$\wS/(\wI+(\Theta)) \cong S/I.$$
(In the present paper, we give a new proof of this fact. See Corollary~\ref{new proof}.)   
Note that the construction of $\BoX(-)$ is different from the standard polarization. 
In fact, it does not give a polarization for a general monomial ideal. 

Moreover
$$
\Theta'= \{ \, x_{i,j} -x_{i+1, j-1} \mid 1 \le i <n,   1< j \le d \, \} \subset \wS
$$
forms an $\wS/\wI$-regular sequence too, and we have  $$\wS/(\wI+(\Theta')) \cong T/I^\sq$$
through $\wS \ni x_{i,j} \longmapsto x_{i+j-1} \in T$ (if we adjust the value of 
$N= \dim T$). The equation $\beta_{i,j}^S(I)=\beta_{i,j}^T(I^\sq)$ mentioned above easily follows from this observation.

In \S2, we will construct a minimal $\wS$-free resolution $\wP_\bullet$ of $\wS/\wI$, 
which is analogous to the  Eliahou-Kervaire resolution of $S/I$. 
However, their description can {\it not} be lifted to $\wI$, and we need modification. 
Roughly speaking,  for $\m \in G(I)$ and $i$ with $i < \max \{ \, l  \mid \text{$x_l$ divides $\m$}\, \} =:\nu(\m)$, 
we use the operation $\m \longmapsto (x_i/x_k)\cdot \m$ 
(and take $\BoX(-)$ of both sides), where $k =\min \{ \, l > i \mid \text{$x_l$ divides $\m$}\, \}$. 
Recall that  Eliahou and Kervaire use  the operation $\m \longmapsto (x_i/x_{\nu(\m)})\cdot \m$. 
Clearly, $\wP_\bullet \otimes_{\wS} \wS/(\Theta)$ and  $\wP_\bullet \otimes_{\wS} \wS/(\Theta')$ 
give the minimal free resolutions of $S/I$ and $T/I^\sq$ respectively. 
We also remark that the method of Herzog and Takayama (\cite[Theorem~1.12]{HT}) is not applicable to our case, 
while $\wI$ is a variant of the ideals treated there. 

Under the assumption that a Borel fixed ideal $I$ is generated in one degree (i.e., all elements of $G(I)$ have the same degree), 
Nagel and Reiner \cite{NR} constructed the alternative polarization $\wI=\BoX(I)$ of $I$, and described  
a minimal $\wS$-free resolution of $\wI$ explicitly (and induced minimal free resolution of $I$ itself and $I^\sq$). 
More precisely, they gave a polytopal CW complex supporting a minimal free resolution of $\wI$.  
In \S4, we will show that their resolution is equivalent to our description. 
In this sense,  our results are generalizations of those for \cite{NR}.

Batzies and Welker (\cite{BW}) developed a strong theory which tries to construct minimal  
free resolutions of monomial ideals using Forman's {\it discrete Morse theory} (\cite{F}). 
If a monomial ideal $J$ is {\it shellable} in the sense of \cite{BW} (i.e., has {\it linear quotients}, 
in the sense of \cite{HT}), their method is applicable to $J$, and we can get a {\it Batzies-Welker type} minimal free resolution.   
Their resolution depends on the choice of an {\it acyclic matching} on $2^{G(J)}$, and the matchings are far from unique in general.  
For most matchings, it is almost impossible to compute the differential map explicitly.   

A Borel fixed ideal $I$ and the polarization $\wI=\BoX(I)$ are shellable.  
In \S5, we will show that our resolution  $\wP_\bullet$ of $\wS/\wI$ and 
the induced resolutions of $S/I$ and $T/I^\sq$ are Batzies-Welker type. 
In particular, these resolutions are cellular. 
As far as the authors know, an {\it explicit} description of a Batzies-Welker type resolution 
of a general Borel fixed ideal has never been obtained before.

Theoretically, \S3, in which we show directly that $\wP_\bullet$ is a resolution, 
is unnecessary, and \S5 is enough.  
However the technique developed in \cite{BW} is quite different from the usual one in this area. 
If we give only the proof based on \cite{BW}, the paper might become unreadable to novice readers. 
Therefore we keep \S3. 

If $I$ is generated in one degree, the CW complex supporting our resolution $\wP_\bullet$ is regular. 
In fact, it coincides with the CW complex of Nagel and Reiner. 
We strongly believe that our CW complex is regular in general, but there is no way to prove it now.

\section{The Eliahou-Kervaire type resolution of \texorpdfstring{$\wS/\BoX(I)$}{b-pol}}
Throughout the rest of the paper, $I$ is a Borel fixed monomial ideal with $\deg \m \le d$ for all $\m \in G(I)$. 
For the definitions of the alternative polarization $\BoX(I)$ of $I$ and related concepts, consult the previous section.  
For a monomial $\m = \prod_{i=1}^n x_i^{a_i} \in S$, set $\mu(\m):= \min \{ \, i \mid a_i > 0\, \}$ and $\nu(\m):= \max\{ \, i \mid a_i > 0\, \}$. 
In \cite{EK}, it is shown that any monomial $\m \in I$ 
has a unique expression $\m= \m_1 \cdot \m_2$ with $\nu(\m_1) \le \mu(\m_2)$ and $\m_1 \in G(I)$.  
Following \cite{EK}, we set $g(\m):=\m_1$. 

For a monomial $\m \in S$ and $i$ with $i < \nu(\m)$, set   
$$\fb_i(\m)=(x_i/x_k) \cdot \m, \ \text{where} \  k:= \min \{ \, j  \mid a_j > 0, \, j > i  \}.$$  
Since $I$ is Borel fixed,  $\m \in I$ implies $\fb_i(\m) \in I$.  

\begin{dfn}\label{admissible2} 
For a finite subset $\wF=\{ \, (i_1,j_1), (i_2, j_2), \ldots, (i_q, j_q) \, \}$ of $\NN \times \NN$ 
and a monomial $\m =\prod_{i=1}^e x_{\alpha_i} = \prod_{j=1}^n x_j^{a_j} \in G(I)$ with $1 \le \alpha_1 \le \alpha_2 \le \cdots \le \alpha_e \le n$, 
we say the pair $(\wF, \wm)$ is {\it admissible} (for $\BoX(I)$), if the following 
conditions are satisfied: 
\begin{itemize}
\item[(a)] $1 \le i_1 < i_2 < \cdots < i_q < \nu(\m)$, 
\item[(b)] $j_r = \max \{ \, l \mid \alpha_l \le i_r\, \} +1$ (equivalently, $j_r=1+\sum_{l=1}^{i_r} a_l$) for all $r$. 
\end{itemize}
For $\m \in G(I)$, the pair $(\emptyset, \wm)$ is also admissible. 
\end{dfn}

The following lemma is easy, and we omit the proof.  

\begin{lem}\label{rem for adm}
Let $(\wF, \wm)$ be an admissible pair with $\wF=\{ \, (i_1,j_1), \ldots, (i_q, j_q) \, \}$ 
and $\m=\prod x_i^{a_i} \in G(I)$. Then we have the following.
\begin{itemize}
\item[(i)] $j_1 \le j_2 \le \cdots \le j_q$.
\item[(ii)] If $r < \nu(\m)$, then $x_{k, j_r} \cdot \BoX(\fb_{i_r}(\m))= x_{i_r, j_r} \cdot \BoX(\m)$, where 
$k=\min \{ \, l \mid l> i_r, a_l > 0 \, \}$.  
\end{itemize}
\end{lem}

For $\m \in G(I)$ and an integer $i$ with $1 \le i < \nu(\m)$, 
set $\m_{\< i\>}:= g(\fb_i(\m))$ and $\wm_{\< i\>}:=\BoX(\m_{\<i\>})$. 
If $i \ge \nu(\m)$, we set $\m_{\< i \>}:=\m$ and $\wm_{\< i\>} := \wm$
for convenience.

In the situation of Lemma~\ref{rem for adm} (ii), $\wm_{\<i_r\>}$ divides $x_{i_r,j_r} \cdot \wm$ 
for all $1 \le r \le q$. 

\begin{lem}\label{m_i_r}
With the above notation, $\m_{\<i\>}$ and $\fb_{i}(\m)$ have the same exponents 
in the variables $x_k$ with $k \le i$. In particular, $\nu(\m_{\<i\>}) \ge i$. 
\end{lem}

\begin{proof}
If the assertion does not hold,  then $\m_{\<i\>}$ divides $\m$, 
and it contradicts the assumption that $\m \in G(I)$. 
\end{proof}

For $\wF=\{ \, (i_1,j_1), \ldots, (i_q, j_q) \, \}$ and $r$ with $1 \le r \le q$, 
set $$\wF_r := \wF \setminus \{ \, (i_r, j_r) \, \}.$$ 
If $(\wF, \m)$ is an admissible pair for $\BoX(I)$, we set 
$$B(\wF, \wm) := \{ \, r \mid \text{$(\wF_r, \wm_{\<i_r\>})$ is admissible}\, \}.$$

\begin{exmp}
Let $I \subset S=\kk[x_1,x_2,x_3,x_4]$ be the smallest Borel fixed ideal containing $\m=x_1^2x_3x_4$. 
In this case, $\m'_{\<i\>}=\fb_{i}(\m')$ for all $\m' \in G(I)$. 
Hence, we have  $\m_{\<1\>} =x_1^3x_4$, $\m_{\<2\>}=x_1^2x_2x_4$ and 
$\m_{\<3\>} =x_1^2x_3^2$. 
Clearly, 
$$(\wF, \wm) =(\{\,(1,3), (2,3), (3,4) \, \}, x_{1,1} \, x_{1,2} \, x_{3,3} \, x_{4,4})$$ is admissible. 
(For this $\wF$, $i_r=r$ holds and the reader should be careful.) Now  
\begin{align*}
& (\wF_1, \wm_{\<1\>})= (\{\,(2,3), (3,4) \, \}, x_{1,1} \, x_{1,2} \, x_{1,3} \, x_{4,4}) \\
& (\wF_2, \wm_{\<2\>})= (\{\,(1,3), (3,4) \, \}, x_{1,1} \, x_{1,2} \, x_{2,3} \, x_{4,4}) \\
& (\wF_3, \wm_{\<3\>})= (\{\,(1,3), (2,3) \, \}, x_{1,1} \, x_{1,2} \, x_{3,3} \, x_{3,4}).
\end{align*} 
The last two pairs are admissible. On the other hand, the first one does not satisfy the condition (b) of Definition~\ref{admissible2}.
Hence $B(\wF, \wm)=\{ 2,3 \}$. 

Next let $I'$ be the smallest Borel fixed ideal containing $\m=x_1^2x_3x_4$ and $x_1^2x_2$. 
For $\wF=\{\,(1,3), (2,3), (3,4) \, \}$, $(\wF, \wm)$ is admissible again. 
However $\wm_{\<2\>}=x_1^2x_2$ in this time, and $(\wF_2, \wm_{\<2\>}) =(\{\,(1,3), (3,4) \, \}, x_{1,1} \, x_{1,2} \, x_{2,3} )$ 
is no longer admissible. Hence $B(\wF, \wm)=\{ 3 \}$  for $\BoX(I')$. 
\end{exmp}

\begin{lem}\label{lem for adm}
Let $(\wF, \wm)$ be as in Lemma~\ref{rem for adm}.  
\begin{itemize}
\item[(i)] For all $r$ with $1 \le r \le q$, $(\wF_r, \wm)$ is admissible. 
\item[(ii)] We always have $q \in B(\wF, \wm)$. 
\item[(iii)] Assume that $(\wF_r, \wm_{\<i_r\>})$ satisfies the condition (a) of Definition~\ref{admissible2}.  
Then $r \in B(\wF, \wm)$  if and only if either $j_r < j_{r+1}$ or $r=q$. 
\item[(iv)] For $r,s$ with $1 \le r < s \le q$ and $j_r <j_s$, we have $\fb_{i_r}(\fb_{i_s}(\m))=\fb_{i_s}(\fb_{i_r}(\m))$ 
and hence $(\wm_{\<i_r\>})_{\<i_s\>}=(\wm_{\<i_s\>})_{\<i_r\>}$. 
\item[(v)] For $r,s$ with $1 \le r < s \le q$ and $j_r = j_s$, we have 
$\fb_{i_r}(\m)= \fb_{i_r} (\fb_{i_s}(\m))$ and hence $\wm_{\<i_r\>}= (\wm_{\<i_s\>})_{\<i_r\>}$. 
\end{itemize}
\end{lem}

\begin{proof} (i) Clear.  

(ii) It suffices to show that $(\wF_q, \wm_{\<i_q\>})$ satisfies the condition (a) of Definition~\ref{admissible2}.  
However, this is clear since  $\nu(\m_{\<i_q\>})  \ge i_q > i_{q-1}$ by  Lemma~\ref{m_i_r}.

(iii) By (ii), we may assume that $r < q$. 
It suffices to consider the condition (b) of Definition~\ref{admissible2}. 
Set $k:=\min\{ \, l \mid  l > i_r, a_l > 0 \, \}$ and $\prod x_i^{b_i} := \fb_{i_r}(\m) = (x_{i_r}/x_k) \cdot \m$. 
Since $(\wF_r, \wm_{\<i_r\>})$ satisfies the condition (a), we have $i_q < \nu(\m_{\<i_r\>})$. 
Hence the exponents of the variables $x_l$ in $\m_{\<i_r\>}$ equals $b_l$ for all $l \le i_q$. 
Clearly, $j_r = j_{r+1}$, if and only if  $a_l=0$ for all $l$ with $i_r< l \le i_{r+1}$, if and only if $k > i_{r+1}$. 
Hence $j_r=j_{r+1}$ implies $\sum_{l \le i_{r+1}}b_l= 1+\sum_{l\le i_r}a_l=j_r=j_{r+1}$, 
and $(\wF_r, \wm_{\<i_r\>})$ does not satisfy the condition (b). 
Next we assume that  $j_r <j_{r+1}$. Then we have $\sum_{l \le i_s}b_l= \sum_{l \le i_s} a_l =j_s-1$ for all $s \le q$ with $s \ne r$,  
and $(\wF_r, \wm_{\<i_r\>})$ satisfies (b). 

(iv) Note that if  $j_r < j_s$ then  $a_l>0$ for some $l$ with $i_r< l \le i_s$. 
Assume that $l$ is the smallest among these integers, and set $k:= \min \{ \, m \mid  m > i_s, a_m >0 \, \}$. 
Then  both $\fb_{i_r}(\fb_{i_s}(\m))$ and $\fb_{i_s}(\fb_{i_r}(\m))$ are equal to $(x_{i_r}x_{i_s}/x_lx_k) \cdot \m$. 

(v) Since $a_l=0$ for all $l$ with $i_r< l \le i_s$ by the assumption,  we have  
$$\fb_{i_r} (\fb_{i_s}(\m))= \fb_{i_r}((x_{i_s}/x_k)\cdot \m)= (x_{i_r}/x_{i_s})\cdot 
(x_{i_s}/x_k)\cdot \m=(x_{i_r}/x_k)\cdot \m= \fb_{i_r}(\m),$$ 
where $k=\min\{ \, l \mid  l > i_r, a_l > 0 \, \}$ (note that $i_s < k$ now). 
\end{proof}

For $F=\{  i_1, \ldots, i_q \} \subset \NN$ with $i_1 < i_2 < \cdots < i_q$ and $\m \in I$, 
Eliahou-Kervaire \cite{EK} call the pair $(F,\m)$ admissible for $I$, if $i_q < \nu(\m)$.   
Clearly, there is a unique sequence $j_1, \ldots, j_q$ such that 
$(\wF, \wm)$ is admissible for $\wI$, where $\wF=\{ \, (i_1, j_q), \ldots, (i_q, j_q) \, \}$.  
In this way, there is a one-to-one correspondence between the admissible pairs for $I$ 
and those for $\wI$. 
As the  free summands of the Eliahou-Kervaire resolution of $I$ are indexed by the admissible pairs for $I$, 
the free summands of our resolution of $\wI$ are indexed by the admissible pairs for $\wI$.

We will define a $\ZZ^{n \times d}$-graded chain complex $\wP_\bullet$ of free $\wS$-modules as follows. 
First, set $\wP_0:=\wS$.  
For each $q \ge 1$, we set 
$$A_q := \text{the set of  admissible pairs $(\wF, \wm)$ for $\BoX(I)$ with $\# \wF=q$},$$
and 
$$\wP_q := \bigoplus_{(\wF,\wm) \in A_{q-1}} \wS \, e(\wF, \wm),$$
where $e(\wF,\wm)$ is a basis element with 
$$\deg \left(e(\wF, \wm)\right)=\deg \left(\wm \times \prod_{(i_r, j_r) \in \wF}x_{i_r, j_r} \right) \in \ZZ^{n \times d}.$$ 
We define the $\wS$-homomorphism $\partial : \wP_{q+1} \to \wP_q$ for $q \ge 1$ so that $e(\wF, \wm)$ 
with $\wF=\{ (i_1, j_1), \ldots, (i_q, j_q)\}$ is sent to 
$$\sum_{ 1 \le  r \le q } (-1)^r \cdot  x_{i_r, j_r} \cdot e(\wF_r, \wm) - \sum_{r \in B(\wF, \wm)} 
(-1)^r \cdot \frac{ x_{i_r, j_r} \cdot \wm}{\wm_{\<i_r\>}} \cdot e(\wF_r, \wm_{\<i_r\>}),$$
and $\partial: \wP_1 \to \wP_0$ by $e(\emptyset, \wm) \longmapsto \wm \in \wS=\wP_0$.  
Clearly, $\partial$ is a $\ZZ^{n\times d}$-graded homomorphism.

Set 
$$\wP_\bullet: \cdots \stackrel{\partial}{\too} \wP_i \stackrel{\partial}{\too}\cdots 
\stackrel{\partial}{\too} \wP_1 \stackrel{\partial}{\too} \wP_0 \too 0.$$
Then we have the following.

\begin{thm}\label{main}
The complex $\wt P_\bullet$ is a $\ZZ^{n\times d}$-graded minimal $\wt S$-free resolution
for $\wt S/\bpol(I)$.
\end{thm}

\section{The proof of Theorem~\ref{main}}
This section is devoted to the proof of Theorem~\ref{main}, but it is not easy to see that $\wP_\bullet$ is even a chain complex. 

\begin{prop}\label{diff}
With the above notation, we have $\partial \circ \partial =0$, that is, $\wP_\bullet$ is a chain complex. 
\end{prop}

\noindent{\it Proof.} 
It suffices to prove $\partial \circ \partial(e(\wF,\wm))=0$ for each admissible pair $(\wF,\wm)$. 
If $\# \wF=1$ (i.e, $\wF=\{(i,j)\}$), then the assertion is easy. In fact, 
\begin{eqnarray*}\partial \circ \partial( \, e(\{ (i,j) \}, \wm) \, ) &=& 
\partial \left( -x_{i,j} \cdot e(\emptyset, \wm) + \frac{x_{i,j} \cdot \wm}{\wm_{\<i\>}} \cdot 
e(\emptyset, \wm_{\<i\>}) \right)\\
&=&-x_{i,j} \cdot \wm + \frac{x_{i,j} \cdot \wm}{\wm_{\<i\>}}\cdot \wm_{\<i\>}\\
&=&0.
\end{eqnarray*}

So we may assume that $q:= \# \wF  >1$.    
For $r,s$ with $0 \le r, s \le q$ and $r \ne s$, set $\wF_{r,s} := \wF \setminus \{\, \{ i_r, j_r\}, \, \{ i_s, j_s\} \, \}$.  
We have a unique expression 
$$\partial \circ \partial(e(\wF,\wm))=C_{r,s}+C,$$
where $C_{r,s}$ (resp. $C$) is an $\wS$-linear combination of the basis elements of the form $e(\wF_{r,s}, -)$ 
(resp. $e(\wF_{t,u}, -)$ with $\{t,u\} \ne \{r,s\}$). 
To show $\partial \circ \partial(e(\wF,\wm))=0$, it suffices to check that $C_{r,s}=0$ for each $r,s$. 
Set
$$\partial(e(\wF,\wm))=C_r+C_s+C',$$
where $C_r$ (resp. $C_s$) is an $\wS$-linear combination of the basis elements of the form $e(\wF_r, -)$ 
(resp. $e(\wF_s,-)$), and $C'$ is an $\wS$-linear combination of $e(\wF_t, -)$ with $t \ne r,s$. 
Then it is easy to see that $C_{r,s}$ is the ``$\wF_{r,s}$-part" of 
$\partial(C_r+C_s)$.  

We will show $C_{r,s}=0$ dividing the arguments into several cases.

\medskip

First, consider the case $r,s \in B(\wF, \wm)$. Then it is clear that $r \in B(\wF_s, \wm)$ and 
$s \in B(\wF_r, \wm)$.  
Since $r, s \in B(\wF, \wm)$, we have $j_r \ne j_s$ by Lemma~\ref{lem for adm} (iii), and hence 
$(\wm_{\<i_s\>})_{\<i_r\>}=(\wm_{\<i_r\>})_{\<i_s\>}$ by Lemma~\ref{lem for adm} (iv). 
Therefore $r \in B(\wF_s, \wm_{\<i_s\>})$ if and only if  $s \in B(\wF_r, \wm_{\<i_r\>})$. 
We have  
\begin{eqnarray*}
\partial(e(\wF,\wm))&=&(-1)^r  \z  \cdot e(\wF_r, \wm) - (-1)^r  \z  \cdot e(\wF_r, \wm_{\<i_r\>})\\
&& +(-1)^s \z  \cdot e(\wF_s, \wm)-(-1)^s \z  \cdot e(\wF_s, \wm_{\<i_s\>})+C',
\end{eqnarray*}
where $C'$ is an $\wS$-linear combination of the basis elements $e(\wF_t, -)$ with $t \ne r,s$, 
and $\z$'s are certain monomials in $\wS$. Of course, each $\z$ are not the same.  
Since $\partial$ is $\ZZ^{n \times d}$-graded, the explicit form of $\z$ is not important. 
Anyway, we will use $\z$ in this meaning in the rest of the proof.  
Without loss of generality, we may assume that $r <s$ in the following computation. 

Assume that $r \in B(\wF_s, \wm_{\<i_s\>})$ (equivalently, $s \in B(\wF_r, \wm_{\<i_r\>})$).  Then 
\begin{eqnarray*}
\partial \circ \partial(e(\wF,\wm))&=&
-(-1)^{r+s}  \z  \cdot e(\wF_{r,s}, \wm) + (-1)^{r+s}  \z  \cdot e(\wF_{r,s}, \wm_{\<i_s\>})\\
& & +(-1)^{r+s}  \z  \cdot e(\wF_{r,s}, \wm_{\<i_r\>})
- (-1)^{r+s}  \z  \cdot e(\wF_{r,s}, (\wm_{\<i_r\>})_{\<i_s\>})\\
& & +(-1)^{r+s} \z  \cdot e(\wF_{r,s}, \wm)-(-1)^{r+s} \z  \cdot e(\wF_{r,s}, \wm_{\<i_r\>})\\
& &-(-1)^{r+s} \z  \cdot e(\wF_{r,s}, \wm_{\<i_s\>})+
(-1)^{r+s} \z  \cdot e(\wF_{r,s}, (\wm_{\<i_s\>})_{\<i_r\>})+C
\end{eqnarray*}
where $C$ is an $\wS$-linear combination of the basis elements $e(\wF_{t,u}, -)$ with $\{t,u\} \ne \{r,s\}$.  
Since $(\wm_{\<i_s\>})_{\<i_r\>}=(\wm_{\<i_r\>})_{\<i_s\>}$ now, $\partial \circ \partial(e(\wF,\m))-C=0$. 

If  $r \not \in B(\wF_s, \wm_{\<i_s\>})$ (equivalently, $s \not \in B(\wF_r, \wm_{\<i_r\>})$), then both  
$$- (-1)^{r+s}  \z  \cdot e(\wF_{r,s}, (\wm_{\<i_r\>})_{\<i_s\>}) \quad  \text{and} 
\quad (-1)^{r+s} \z  \cdot e(\wF_{r,s}, (\wm_{\<i_s\>})_{\<i_r\>})$$ 
do not appear in the above expansion of $\partial \circ \partial(e(\wF,\wm))$.   
Hence $\partial \circ \partial(e(\wF,\wm))-C=0$ remains true. 

\begin{lem}
If $r, s \not \in B(\wF,\wm)$, then we have $r \not \in B(\wF_s, \wm)$ and $s \not \in B(\wF_r, \wm)$. 
\end{lem}

\begin{proof}
It suffices to prove $r \not \in B(\wF_s, \wm)$. For the contrary, assume that 
$(\wF_{r,s}, \wm_{\<i_r\>})$ is admissible, in particular, it satisfies the condition (a) of
Definition~\ref{admissible2}. 
Since $s \ne q$ by Lemma~\ref{lem for adm} (ii), $(\wF_r, \wm_{\<i_r\>})$ also satisfies the condition (a). 
Since $r \not \in B(\wF,\wm)$ now, we have $j_r =j_{r+1}$ by  Lemma~\ref{lem for adm} (iii). 
Hence the assumption $r  \in B(\wF_s, \wm)$ implies that $s=r+1$ and $(j_r =) \, j_{r+1} < j_{r+2}$. 
If $(\wF_s, \wm_{\<i_s\>})$ satisfies  the condition (a) of Definition~\ref{admissible2}, then we have 
$s \in B(\wF,\wm)$ and it contradicts the assumption. 
Hence $(\wF_s, \wm_{\<i_s\>})$ does not satisfy (a), that is, $\nu(\m_{\<i_s\>}) \le i_q$.  
Since $\m_{\<i_r\>}= (\m_{\<i_s\>})_{\<i_r\>}$ by  Lemma~\ref{lem for adm} (v), 
we have $\nu(\m_{\<i_r\>}) \le \nu(\m_{\<i_s\>}) \le i_q$ and $(\wF_{r,s}, \wm_{\<i_r\>})$ 
does not satisfy (a). This is a contradiction.  
\end{proof}

Hence, when $r, s \not \in B(\wF,\m)$,  $\partial \circ \partial(e(\wF,\m))-C=0$ is easy. 
So it remains to prove the case where  $r \not \in B(\wF, \wm)$ but $s \in B(\wF, \wm)$.  

\begin{lem}\label{case 3}
Assume that $r \not \in B(\wF, \wm)$ and $s \in B(\wF, \wm)$. 
Then we always have $s \in B(\wF_r, \wm)$. 
Moreover, $r  \in B(\wF_s, \wm_{\<i_s\>})$ if and only if  $r \in B(\wF_s, \wm)$. 
If this is the case, we have  $\m_{\<i_r\>}= (\m_{\<i_s\>})_{\<i_r\>}$. 
\end{lem}

\begin{proof}
The first assertion is clear,  and we will show the second and the third (at the same time).   
Since $\nu((\m_{\<i_s\>})_{\<i_r\>}) \le \nu(\m_{\<i_r\>})$, $r  \in B(\wF_s, \wm_{\<i_s\>})$ implies  
$r \in B(\wF_s, \wm)$. So,  assuming that $r \in B(\wF_s, \wm)$, we will show that $\m_{\<i_r\>}= 
(\m_{\<i_s\>})_{\<i_r\>}$, which implies $r \in B(\wF_s, \wm_{\<i_s\>})$.  
Since $r \not \in B(\wF, \wm)$, $(\wF_r, \wm_{\<i_r\>})$ 
does not satisfy at least one of conditions (a) and (b) in Definition~\ref{admissible2}. 

\noindent {\it Case 1.} Assume that $(\wF_r, \wm_{\<i_r\>})$ fails (a), 
that is,  $\nu(\m_{\<i_r\>}) \le i_q$. 
Since $r \in B(\wF_s, \wm)$ now, we have $s=q$ and $F_s = \{ i_1,\dots ,i_{q-1} \}$.
Hence $(i_r \le) \, i_{q-1} < \nu(\m_{\<i_r\>})$. 
Since $i_r < \nu(\m_{\<i_r\>}) \le i_q$, we have $j_r < j_q$, and hence 
$\fb_{i_r} (\fb_{i_q}(\m))=\fb_{i_q}(\fb_{i_r}(\m))$. 
Therefore, 
$(\wm_{\<i_q\>})_{\
<i_r\>}=(\wm_{\<i_r\>})_{\<i_q\>}=\wm_{\<i_r\>},$
where the first equality follows from Lemma~\ref{lem for adm} (v), and 
the second follows from $\nu(\m_{\<i_r\>}) \le i_q$. 

\noindent{\it Case 2.} Assume that $(\wF_r, \wm_{\<i_r\>})$ fails (b) of  Definition~\ref{admissible2}. 
By Lemma~\ref{lem for adm} (iii), we have $s=r+1$ and $j_r=j_s$. 
Hence $\m_{\<i_r\>}= (\m_{\<i_s\>})_{\<i_r\>}$ by Lemma~\ref{lem for adm} (v). 
\end{proof}

\noindent{\it The continuation of the proof of Proposition~\ref{diff}.} 
Finally, consider the case  $r \not \in B(\wF, \wm)$ and $s \in B(\wF, \wm)$. Then we have 
$$\partial(e(\wF,\wm))=(-1)^r  \z  \cdot e(\wF_r, \wm)+(-1)^s \z  \cdot e(\wF_s, \wm)-(-1)^s 
\z  \, e(\wF_s, \wm_{\<i_s\>})+C',$$
where $C'$ is an $\wS$-linear combination of the elements $e(\wF_t, -)$ with $t \ne r,s$. 

Assume that  $r \in B(\wF_s, \wm_{\<i_s\>})$ (equivalently, $r \in B(\wF_s, \wm)$). 
Then we have $r < s$ as shown in the proof of Lemma~\ref{case 3}, and 
\begin{eqnarray*}
\partial \circ \partial(e(\wF,\wm))
&=&-(-1)^{r+s}  \z  \cdot e(\wF_{r,s}, \wm) + (-1)^{r+s}  \z  \cdot e(\wF_{r,s}, \wm_{\<i_s\>})\\
& & +(-1)^{r+s} \z  \cdot e(\wF_{r,s}, \wm)-(-1)^{r+s} \z  \cdot e(\wF_{r,s}, \wm_{\<i_r\>})\\
& &-(-1)^{r+s} \z  \cdot e(\wF_{r,s}, \wm_{\<i_s\>})
+(-1)^{r+s} \z  \cdot e(  \wF_{r,s}, (\wm_{\<i_s\>})_{\<i_r\>} ) +C,
\end{eqnarray*}
where $C$ is an $\wS$-linear combination of the basis elements $e(\wF_{t,u}, -)$ with $\{t,u\} \ne \{r,s\}$.  
Since $\m_{\<i_r\>}=(\m_{\<i_s\>})_{\<i_r\>}$ by Lemma~\ref{case 3}, 
we have  $\partial \circ \partial(e(\wF,\wm))-C=0$. 

Next, assume that  $r \not \in B(\wF_s, \wm_{\<i_s\>})$ (equivalently, $r \not \in B(\wF_s, \wm)$). 
Then both  
$$-(-1)^{r+s}  \z  \, e(\wF_{r,s}, \wm_{\<i_r\>}) \quad  \text{and} 
\quad (-1)^{r+s} \z  \cdot e( \wF_{r,s}, (\wm_{\<i_s\>})_{\<i_r\>} )$$ 
do not appear in the above expansion of 
$\partial \circ \partial(e(\wF,\wm))$.   
Hence $\partial \circ \partial(e(\wF,\wm))-C=0$ remains true. 
\qed

\medskip

Next, we will show that $\wP_\bullet$ is acyclic. To do so, we need some preparation.

\newcommand\nvert{\ {\ooalign{$\slash$\crcr\hss$\vert$\hss}}\ }

Let $\succ$ be the lexicographic order on the monomials of $S$ with
$x_1 \succ x_2 \succ \cdots \succ x_n$.  
In the rest of this section, $\succ$ means this order.  
In the sequel, the minimal monomial generators of the Borel fixed ideal $I$ are denoted by $\m_1, \dots, \m_t$ with $\m_1 \succ \m_2 \succ \cdots \succ \m_t$.

Let $I_r$ denote the ideal generated by $\m_1,\dots, \m_r$.

\begin{lem}\label{sec:bfilt}
Each $I_r$ is also Borel fixed.
\end{lem}
\begin{proof}
We use induction on $r$.
It suffices to show that for $i, j$ with $1 \le i < j$, $x_j \mid \m_r$ implies $(x_i /x_j) \m_r \in I_r$.
Set $\m_r' := (x_i/x_j)\m_r$.
Since $I$ is Borel fixed, $\m_r'\in I$.
By the similar argument as in the proof of Lemma~\ref{m_i_r},
we see that for $k \le i$ the exponent of $x_k$ in $g(\m_r')$
coincides with that in $\m_r'$.
Thus the exponent of $x_k$ in $g(\m_r')$ is equal to that in $\m_r$ if $k < i$ and is more than
if $k = i$. This implies that $g(\m_r') \succ \m_r$ and hence $g(\m_r') \in I_r$.
Therefore $\m_r' \in I_r$.
\end{proof}

For simplicity, we set  $\wt I_r := \bpol(I_r) (\subset \wt S)$.

\begin{lem}\label{sec:bpolfilt}
Let $\m_r = \prod_{k=1}^nx_k^{a_k}$ and set $b_i := 1 + \sum_{k=1}^i a_k$.
Then it follows that $(\wt I_{r-1}: \wt \m_r) = (x_{1,b_1},\dots ,x_{\gl,b_\gl})$, where $\gl := \nu(\m_r)-1$.
\end{lem}
\begin{proof}
Since $\wm_r$ is a monomial and $\wt I_{r-1}$ is a monomial ideal, we have
$$
(\wt I_{r-1} : \wm_r) =\left( \frac{\lcm(\wm_1,\wm_r)}{\wm_r},\dots ,\frac{\lcm(\wm_{r-1},\wm_r)}{\wm_r} \right).
$$
Let $s$ be an integer with $1\le s \le r-1$. Since $\m_s \succ \m_r$,
there exists an integer $l$ such that the exponent of $x_k$ in $\m_s$
is equal to that in $\m_r$ for $k < l$ and is more than if $k = l$.
Thus $x_{l,b_l} \mid \wm_s$ and $x_{l,b_l} \nmid \wm_r$.
Hence $x_{l,b_l} \mid \left(\lcm(\wm_s ,\wm_r)/\wm_r \right)$.
Since $\m_s, \m_r \in G(I)$, $l$ must be less than $\nu(\m_r)$; otherwise
$\m_r \mid \m_s$, which is a contradiction.
Hence it follows that $\lcm(\wm_s,\wm_r) /\wm_r \in (x_{1,b_1},\dots ,x_{\gl,b_\gl})$,
and therefore $(\wt I_{r-1}: \wt \m_r) \subset (x_{1,b_1},\dots ,x_{\gl,b_\gl})$.

Now let us consider the inverse inclusion.
Let $s$ be the integer with $1 \le s \le \gl$. Then there exists $j$ with $j > s$
such that $a_j >0$.
By Lemma~\ref{rem for adm} (ii), we have $(\wm_r)_{\< s\>} \mid x_{s,b_s}\wm_r$.
On the other hand, Lemma~\ref{m_i_r} implies $(\m_r)_{\< s\>} \succ \m_r$,
and hence $(\wm_r)_{\< s\>} \in \wt I_{r-1}$.
It follows that $x_{s,b_s}\wm_r \in \wt I_{r-1}$, or equivalently $x_{s,b_s} \in (\wt I_{r-1}: \wt \m_r)$.
Therefore we conclude that $(\wt I_{r-1}: \wt \m_r) \supseteq (x_{1,b_1},\dots ,x_{\gl,b_\gl})$.
\end{proof}

The differential $\diff$ of the complex $\wt P_\bullet$, constructed in the above,
consists of the following 2 maps.
Let $\gd,\gd': \wt P_\bullet \to \wt P_\bullet$ be the family of $\wt S$-homomorphism
$\gd_{q+1}, \gd'_{q+1}:\wt P_{q+1} \to \wt P_q$ which sends $e(\wt F, \wm)$ 
with $\wF=\{ (i_1, j_1), \ldots, (i_q, j_q)\}$ and $q \ge 1$ to 
\begin{align*}
&\gd_{q+1}(e(\wF, \wm)) = \sum_{1 \le  r \le q } (-1)^r \cdot  x_{i_r, j_r} \cdot e(\wF_r, \wm), \\
&\gd'_{q+1}(e(\wF, \wm)) = \sum_{r \in B(\wF, \wm)}(-1)^r \cdot \frac{ x_{i_r, j_r} \cdot \wm}{\wm_{\<i_r\>}} \cdot e(\wF_r, \wm_{\<i_r\>}),
\end{align*}
respectively. In the case $q=0$, we set $\gd(e(\void ,\wm)) = 0$ and $\gd'(e(\void,\wm)) = - \wm$.
Then it follows that $\diff = \gd - \gd'$.

To show that $\wt P_\bullet$ is acyclic, we make use of the technique by means of mapping cones as the one by Peeva and Stillman in \cite{PS}. Let $(U_\bullet, \diff^U),(V_\bullet, \diff^V)$ be a complex and
$f:U_\bullet \to V_\bullet$ be a homomorphism of complexes.
The {\it mapping cone} $\Con_\bullet(f)$ of $f$ is a complex such that $\Con_q(f) := V_q \oplus U_{q-1}$
for all $q$ and the differential $\diff^{\Con(f)}$ are defined by
$\diff^{\Con(f)}(v,u) = (\diff^V(v) + f(u), -\diff^U(u))$ for $(v,u) \in \Con_q(f)$.
It is noteworthy that $V_\bullet$ is then a subcomplex of $\Con_\bullet(f)$.

\bigskip

\begingroup
\renewcommand{\proofname}{The proof of Theorem~\ref{main}.}
\begin{proof}
We use induction on $t = \# G(I)$. The case $t = 1$ is trivial. Assume $t > 1$.
Let $\wt I, \wt I_r, \wt \m_r$, where $1 \le r \le \# G(I)$, be as above.
There is the following exact sequence
$$
0 \longto \left(\wt S/(\wt I_{t-1}:\wt \m_t)\right) (-\wt\ba) \overset{\wt \m_t}{\longto} \wt S/\wt I_{t-1} \longto
\wt S/\wt I \longto 0,
$$
where $\wt\ba$ denotes the multidegree of $\wm_t$.
By Lemma~\ref{sec:bpolfilt}, $\wt I_{t-1}:\wt \m_t = (x_{1,b_1},\dots ,x_{\gl,b_\gl})$ for suitable $\gl$ and $b_1,\dots b_\gl$,
and the Koszul complex 
$K_\bullet(x_{1,b_1},\dots ,x_{\gl,b_\gl})$ of the sequence $x_{1,b_1},\dots ,x_{\gl,b_\gl}$
is a minimal free resolution of $\wt S/(\wt I_{t-1} :\wt \m_t)$.
Set $\wt K_\bullet := \left(K_\bullet(x_{1,b_1},\dots ,x_{\gl,b_\gl})\right)(-\wt \ba)$.
Assigning each basis $x_{i_1,b_{i_1}}\wedge \cdots \wedge x_{i_q,b_{i_q}}$
of $\wt K_q$ to the one $e(\{ (i_1, b_{i_1}), \ldots, (i_q, b_{i_q})\},\wm_t)$
of $\wt P_{q+1}$, we have the injective map
$\wt K_q \embto \wt P_{q+1}$ for each $q$.
Via this map, the differential map of $\wt K_\bullet$ coincides with $\gd$.
By the inductive hypothesis, we obtain the minimal free resolution $\wt Q_\bullet$ of $\wt S/\wt I_{t-1}$
constructed in the same way as $\wt P_\bullet$.
The complex $\wt Q_\bullet$ is naturally regarded as a subcomplex of $\wt P_\bullet$, since an admissible pair $(\wF, \wm)$ for $\wt I_{t-1}$ is also a one for $\wt I_t$
and since the set $B(\wF,\wm)$ in $\wt I_{t-1}$ coincides with that in $\wt I_t$.
It is easy to verify that after regarding each $\wt K_q$ as a direct summand of $\wt P_{q+1}$, the map $-\gd'$ is then a homomorphism of complexes
from $\wt K_\bullet$ to $\wt Q_\bullet$ and is a lifting of
the map $\left(\wt S/(\wt I_{t-1}:\wt \m_t)\right) (-\wt\ba) \to \wt S/\wt I_{t-1}$ sending $1$ to $\wm_t$.

Now consider the mapping cone $\Con_\bullet(-\gd')$ of $-\gd'$.
It coincides with $\wt P_\bullet$. From a long exact sequence of homologies
induced by the one
$$
0 \longto \wt Q_\bullet \longto \Con_\bullet(-\gd') \longto \wt K_\bullet \bra{1} \longto 0,
$$
it follows that $\wt P_\bullet$ is a minimal $\ZZ^{n\times d}$-graded free resolution of $\wt S/\bpol(I)$, since
$$
H_0(\wt K_\bullet) \cong \left(\wt S/(\wt I_{t-1}:\wt \m_t)\right) (-\wt\ba)
\stackrel{\wm_t}{\longto} \wt S/\wt I_{t-1} \cong H_0(\wt Q_\bullet)
$$
is injective.
\end{proof}
\endgroup

\begin{rem}\label{regular decomp}
In their paper \cite{HT}, Herzog and Takayama explicitly gave a minimal free resolution of a monomial ideal with  
{\it linear quotients} admitting a {\em regular decomposition function}
(see \cite{HT} for the definitions). A Borel fixed ideal $I$ is a typical example with this property. 
However, while our  $\wI$ has  linear quotients, the decomposition function $\wg$ can not be  regular in general. 
For example, if  $I=(x^3, x^2y, xy^2, y^3)$, then 
$\wI=(x_1x_2x_3, x_1x_2y_3,x_1y_2y_3,  y_1y_2y_3)$ has linear quotients in several total orders. 
Checking each case, we see that $\wg$ is not regular for every order. 
Here we check this in two typical cases. First, consider the order $x_1x_2x_3 > x_1x_2y_3 > x_1y_2y_3 > y_1y_2y_3$, 
which gives linear quotients. We have $(x_1x_2x_3, x_1x_2y_3) : (x_1y_2y_3) =(x_2)$, $\wg(x_2(x_1y_2y_3))=x_1x_2y_3$, 
and $(x_1x_2x_3):(x_1x_2y_3)=(x_3) \not \subset (x_2)$. It means that $\wg$ is not regular in this case.  
Next, consider the order $x_1x_2y_3 > x_1x_2x_3  > x_1y_2y_3 > y_1y_2y_3$ which also gives  linear quotients. 
$(x_1x_2y_3, x_1x_2x_3, x_1y_2y_3): (y_1y_2y_3) = (x_1)$, $\wg(x_1(y_1y_2y_3))=x_1y_2y_3$ and 
$(x_1x_2y_3, x_1x_2x_3) :(x_1y_2y_3)=(x_2) \not \subset (x_1)$. Hence $\wg$ is not regular again. 

Anyway, we can not directly apply \cite{HT} to our case.  
\end{rem}

\section{Applications and Remarks}
Let $I \subset S$ be a Borel fixed ideal, and $\Theta \subset \wS$ the sequence defined in Introduction. 
In \cite{y10}, the second author has shown that $\wI:= \BoX(I)$ gives a polarization of $I$, that is, 
$\Theta$ forms an $(\wS/\wI)$-regular sequence. The proof there is rather direct, and uses the sequentially 
Cohen-Macaulay property of $\wS/\wI$. However we can give a new proof now.   

Recall that Eliahou and Kervaire \cite{EK} constructed a minimal $S$-free resolution of $I$, and free 
summands of their resolution are indexed by the admissible pairs for $I$. Here an admissible 
pair for $I$ is a pair $(F,\m)$ such that $\m \in G(I)$ and $F=\emptyset$ or 
$F = \{ i_1, \ldots ,i_q \} \subset \NN$ with $1 \le i_1 < \cdots < i_q < \nu(\m)$.
As remarked after the proof of Lemma~\ref{lem for adm}, 
there is a one-to-one correspondence between the admissible pairs for $\wI$ and those for $I$, 
and if $(\wF,\wm)$ corresponds to $(F,\m)$ then $\#\wF=\#F$. 
Hence we have 
\begin{equation}\label{Betti num}
\beta_{i,j}^{\wS}(\wt I) = \beta_{i,j}^S(I) 
\end{equation}
for all $i,j$, where $S$ and $\wS$ are considered to be $\ZZ$-graded by setting the degree of all variables to be $1$.

Of course, this equation is clear, if we know the fact that $\wI$ is a polarization of $I$ 
(\cite[Theorem~3.4]{y10}). Conversely, this equation induces the theorem. 

\begin{cor}[{\cite[Theorem 3.4]{y10}}]\label{new proof}
The ideal $\wI$ is a polarization of $I$.
\end{cor}

\begin{proof}
Follows from the equation~\eqref{Betti num} and \cite[Lemma~6.9]{NR} (see also \cite[Lemma 2.2]{y10}). 
\end{proof}

The next result also follows from \cite[Lemma~6.9]{NR}. 

\begin{cor}
$\wP_\bullet \otimes_{\wS} \wS/(\Theta )$ is a minimal $S$-free resolution of $S/I$. 
\end{cor}

\begin{rem}\label{rem:diff res}
(1) The correspondence between the admissible pairs for $I$ and those for $\wI$, does {\it not} give
a chain map between the Eliahou-Kervaire resolution and our $\wP_\bullet \otimes_{\wS} \wS/(\Theta )$.
Let us give an example to verify this. In the following argument, refer to \cite{EK}
for the definition of  differential maps due to Eliahou-Kervaire.
Let $I := (x_1^2,x_1x_2, x_1x_3, x_1x_4, x_2^2, x_2x_4) \subset \kk[x_1,x_2,x_3,x_4]$. Then $(\{ 1,2,3 \}, x_2x_4)$ is an admissible pair
in the sense of Eliahou-Kervaire, and $(\{ (1,1), (2,2), (3,2) \}, x_{2,1}x_{4,2})$ is the corresponding admissible one in our sense.
In the Eliahou-Kervaire resolution, the image of the basis $e(\{ 1, 2, 3\}, x_2x_4)$ corresponding to $(\{ 1,2,3 \}, x_2x_4)$,
by the differential map, is a linear combination (with nonzero coefficients in $S$) of
the following four bases
$$
e(\{ 2, 3 \}, x_2x_4), \ e(\{ 1, 3\} ,x_2x_4),\ e(\{ 1, 2 \}, x_2x_4), \ e(\{ 1, 2 \}, x_2x_3).
$$
On the other hand, the image of $e(\{ (1,1), (2,2), (3,2)\} ,x_{2,1}x_{4,2})$ in our resolution is
a linear combination of these five bases
\begin{align*}
&e(\{ (2,2), (3,2)\}, x_{2,1}x_{4,2}),\ e(\{ (1,1), (3,2)\} ,x_{2,1}x_{4,2}), \ e(\{ (1,1), (2,2)\} ,x_{2,1}x_{4,2}) \\
&e(\{ (1,1), (2,2) \} ,x_{2,1}x_{3,2}), \ e(\{ (2,2), (3,2)\} ,x_{1,1}x_{4,2}).
\end{align*}
Thus the correspondence between admissible pairs for $I$ and $\wI$ does not give a chain map.

In this example, the image of $e(\{ 1, 2, 3\}, x_2x_4)$ is a linear combination of four base elements, 
while that of  $e(\{ (1,1), (2,2), (3,2)\} ,x_{2,1}x_{4,2})$  needs five base elements. 
It is also easy to find an ``opposite" example, that is, the image of  $e(\wF, \wm)$ in our resolution sometimes requires fewer bases elements 
than that of $e(F,\m)$ in the Eliahou-Kervaire resolution.

In these senses, our resolution differs from the one due to Eliahou-Kervaire resolution in general.
In Example~\ref{final remark} below,
a difference between two resolutions can be recognized visually.

(2) Eliahou and Kervaire (\cite{EK}) constructed minimal free resolutions of {\it stable monomial ideals}, 
which form a wider class than Borel fixed ideals. As shown in \cite[Example~2.3 (2)]{y10}, 
$\BoX(J)$ is not a  polarization for  a stable monomial ideal $J$ in general, and 
the construction of $\wP_\bullet \otimes_{\wS} \wS/(\Theta )$ does not work for $J$.

(3) By the equation~\eqref{Betti num} and Corollary~\ref{new proof}, one might expect that
$\BoX(-)$ preserves {\it lcm-lattices}.  But it does not in general.
Recall that the lcm-lattice of a monomial ideal $J$ is the set 
$\LCM (J) := \{ \, \lcm \{ \, \m  \mid \m \in \gs \, \} 
\mid \gs \subset G(J) \, \}$ with the order given by divisibility. Clearly,  $\LCM(J)$ forms a lattice.
Let $\vee$ denote the join in the lcm-lattice.
For the Borel fixed ideal $I = (x^2, xy, xz, y^2, yz)$, 
we have $xy \vee xz = xy \vee yz = xz \vee yz$ in $\LCM(I)$. On the other hand,  
$\wt{xy} \vee \wt{xz} = x_1y_2z_2$, $\wt{xy} \vee \wt{yz} = x_1y_1y_2z_2$ and $\wt{xz} \vee \wt{yz} = x_1y_1z_2$ 
are all distinct in $\LCM (\wI )$ .
\end{rem}

Let $a = \{a_0,a_1,a_2,\dots\}$ be a non-decreasing sequence of non-negative integers with $a_0 = 0$, and $T = \kk[x_1,\dots ,x_N]$ a polynomial ring with $N \gg 0$.
In his paper \cite{Mu}, Murai defined an operator $(-)^{\gamma(a)}$
acting on monomials and monomial ideals of $S$.
For a monomial $\m \in S$ with the expression $\m = \prod_{i=1}^e x_{\alpha_i}$
as \eqref{alpha expression}, set
$$
\m^{\gamma(a)} := \prod_{i=1}^e x_{\ga_i + a_{i-1}} \in T,
$$
and for a monomial ideal $I \subset S$,
$$
I^{\gamma(a)} := (\m^{\gamma(a)} \mid \m \in G(I) ) \subset T.
$$
If $a_{i+1} > a_i$ for all $i$, then $I^{\gamma(a)}$ is a squarefree monomial
ideal. Particularly, in the case $a_i = i$ for all $i$, the operator
$(-)^{\gamma(a)}$ is just the squarefree operator $(-)^\sq$
due to Kalai,
which plays an important role in the construction of the {\em symmetric shifting} of a simplicial complex (see \cite{AHH2}).

The operator $(-)^{\gamma(a)}$ also can be described by $\bpol(-)$
as is shown by the second author (\cite{y10}).
Let $L_a$ be the $\kk$-subspace of $\wS$ spanned by $\{x_{i,j} - x_{i',j'} \mid i + a_{j-1} = i' + a_{j'-1} \}$, and 
$\Theta_a$ a basis of $L_a$. For example, we can take
$$
\{\ x_{i,j} - x_{i+1,j-1} \mid 1 \le i < n,\ 1 < j \le d\ \}
$$
as $\Theta_a$ in the case $a_i = i$ for all $i$.
With a suitable choice of the number $N$, the ring homomorphism $\wS \to T$ with $x_{i,j} \mapsto x_{i+a_{j-1}}$ 
induces the isomorphism $\wS/(\Theta_a) \cong T$.

\begin{prop}[{\cite[Proposition 4,1]{y10}}]\label{sqf_operator}
With the above notation,  $\Theta_a$ forms an $\wS / \wI$-regular sequence, and we have
$(\wS /(\Theta_a)) \otimes_{\wS} (\wS /\wI) \cong T/I^{\gamma (a)}$
through the isomorphism $\wS/(\Theta_a) \cong T$.
\end{prop}

Applying Proposition~\ref{sqf_operator} and \cite[Proposition 1.1.5]{BH}, we have the following.

\begin{cor}
The complex $\wP_\bullet \otimes_{\wS} \wS/(\Theta_a)$
is a minimal $T$-free resolution of $T/I^{\gamma(a)}$. 
In particular, a minimal free resolution of $T/I^\sq$ is given in this way. 
\end{cor}

To draw a ``diagram" of an admissible pair $(\wF,\wm)$, we put a white square in the  $(i,j)$-th position if $(i,j) \in \wF$ and a black square there 
if $x_{i,j}$ divides $\wm$. If $\wF$ is maximal among $\wF'$ such that $(\wF',\wm)$ is admissible,  
then the diagram of $(\wF,\wm)$ forms a ``right side down stairs". 
For a non-maximal $\wF$, some white squares are removed from the diagram for the maximal case.    
Figure~1 (resp. Figure~2) below is a diagram of $(\wF,\wm)$  with $\wm=x_{1,1}x_{1,2}x_{2,3}x_{6,4}x_{6,5}$ and 
$\wF=\{ \, (1,3), (2,4), (3,4), (4,4),(5,4) \, \}$ (resp. $\wF=\{ \, (1,3), (3,4), (4,4) \, \}$). 
Clearly, the former is a maximal case. 

\begin{figure}[h]
\setlength\unitlength{.4mm}
\begin{minipage}{.4\textwidth}
\begin{center}
\begin{picture}(80,80)(0,0)
\thicklines
\put(20,50){\shade\path(0,0)(0,10)(10,10)(10,0)(0,0)}
\put(30,50){\shade\path(0,0)(0,10)(10,10)(10,0)(0,0)}
\put(40,50){\whiten\path(0,0)(0,10)(10,10)(10,0)(0,0)}
\put(40,40){\shade\path(0,0)(0,10)(10,10)(10,0)(0,0)}
\put(50,40){\whiten\path(0,0)(0,10)(10,10)(10,0)(0,0)}
\put(50,30){\whiten\path(0,0)(0,10)(10,10)(10,0)(0,0)}
\put(50,20){\whiten\path(0,0)(0,10)(10,10)(10,0)(0,0)}
\put(50,10){\whiten\path(0,0)(0,10)(10,10)(10,0)(0,0)}
\put(50,0){\shade\path(0,0)(0,10)(10,10)(10,0)(0,0)}
\put(60,0){\shade\path(0,0)(0,10)(10,10)(10,0)(0,0)}
\put(0,0){\makebox(5,60){$i$}}
\put(18,77){\makebox(60,5){$j$}}
\put(10,3){$6$}
\put(10,13){$5$}
\put(10,23){$4$}
\put(10,33){$3$}
\put(10,43){$2$}
\put(10,53){$1$}
\put(23,66){$1$}
\put(33,66){$2$}
\put(43,66){$3$}
\put(53,66){$4$}
\put(63,66){$5$}
\end{picture}
\end{center}
\caption{}
\end{minipage}
\begin{minipage}{.4\textwidth}
\begin{center}
\begin{picture}(80,80)(0,0)
\thicklines
\put(20,50){\shade\path(0,0)(0,10)(10,10)(10,0)(0,0)}
\put(30,50){\shade\path(0,0)(0,10)(10,10)(10,0)(0,0)}
\put(40,50){\whiten\path(0,0)(0,10)(10,10)(10,0)(0,0)}
\put(40,40){\shade\path(0,0)(0,10)(10,10)(10,0)(0,0)}
\put(50,30){\whiten\path(0,0)(0,10)(10,10)(10,0)(0,0)}
\put(50,20){\whiten\path(0,0)(0,10)(10,10)(10,0)(0,0)}
\put(50,0){\shade\path(0,0)(0,10)(10,10)(10,0)(0,0)}
\put(60,0){\shade\path(0,0)(0,10)(10,10)(10,0)(0,0)}
\put(0,0){\makebox(5,60){$i$}}
\put(18,77){\makebox(60,5){$j$}}
\put(10,3){$6$}
\put(10,13){$5$}
\put(10,23){$4$}
\put(10,33){$3$}
\put(10,43){$2$}
\put(10,53){$1$}
\put(23,66){$1$}
\put(33,66){$2$}
\put(43,66){$3$}
\put(53,66){$4$}
\put(63,66){$5$}
\end{picture}
\end{center}
\caption{}
\end{minipage}
\end{figure}

The lower right end of the stairs must be a black square.  
For each $j$ with $1 \le j \le \deg \wm$, there is a unique 
black square in the $j$-th column. The black square in the $j$-th column is the ``lowest" of the squares in this column.

\begin{lem}\label{one generated}
Assume that $I$ is generated in one degree (i.e., all element of $G(I)$ have the same degree).
Let $(\wF, \wm)$ with $\wF=\{ \, (i_1,j_1), \ldots, (i_q, j_q) \, \}$  be an admissible pair for $\BoX(I)$. 
Then, $r \in B(\wF, \wm)$ if and only if $j_r < j_{r+1}$ or $r=q$.  
\end{lem}

\begin{proof}
By Lemma~\ref{lem for adm} (iii), 
it suffices to show that $r < q$ and $j_r < j_{r+1}$ imply $i_q < \nu(\m_{\<i_r\>})$. 
Under this assumption, we have $k:= \min\{ \, l > i_r \mid \text{$x_l$ divides $\m$} \, \} \leq i_{r+1} < \nu(\m)$.   
Since $I$ is generated in one degree, we have 
$\m_{\<i_r\>}= \fb_{i_r}(\m)=(x_{i_r}/x_k) \cdot  \m$ and  $\nu(\m_{\<i_r\>})=\nu(\m)>i_q$. 
\end{proof}

In the rest of this section,  we assume that $I$ is generated in one degree unless otherwise specified. 
Let us go back to the diagram of $(\wF,\wm)$.  
By Lemma~\ref{one generated},  $r \in B(\wF,\wm)$ if and only if  the white square in the $(i_r,j_r)$-th position  
is the ``lowest" one among white squares in the $j_r$-th column. In this case, we can get the diagram of the admissible pair 
$(\wF_r, \wm_{\<i_r\>})$ from that of $(\wF, \wm)$ by the following procedure.

\begin{itemize}
\item[(i)] Remove the (sole) black square in the $j_r$-th column. 
\item[(ii)] Replace the white square in  the $(i_r, j_r)$-th position by a black one. 
\end{itemize} 

After minor modification, the above mentioned relation between the diagram of $(\wF,\wm)$ and that of $(\wF_r, \wm_{\<i_r\>})$ 
remains true for a general Borel fixed ideal. In general, we might have $\m_{\<i_r\>} \ne \fb_{i_r}(\m)$. 
If this is the case,  we have to remove (possibly several) black squares from the right end of the diagram of $(\wF,\wm)$. 

\medskip

For an admissible pair  $(\wF, \wm)$ with $\wF=\{ \, (i_1,j_1), \ldots, (i_q, j_q) \, \}$,  
set $$x(\wF, \wm):= \wm \times \prod_{r=1}^q x_{i_r, j_r} \in \wS. $$
Clearly, the ``support" of $x(\wF,\wm)$ coincides with the all squares (i.e., black and  white squares) in the diagram of $(\wF,\wm)$.
Since the coloring of each square is uniquely determined by
its position in the diagram (the squares on the bottom
of each column are black, and the others white), it follows that
\begin{equation}\label{diagram recovers}
x(\wF, \wm) = x(\wF' ,\wm') \iff (\wF, \wm) = (\wF',\wm')
\end{equation}
for two admissible pairs $(\wF, \wm)$ and $(\wF', \wm')$.

Set
$$\rmv(\wF, \wm, j):= \begin{cases}
\{ \, (i,j) \mid \text{$x_{i,j}$ divides $x(\wF, \wm)$} \, \} & \text{if $j=j_r$ for some $r$,}\\
\emptyset & \text{otherwise,}
\end{cases}$$
and  
\begin{equation}\label{rm decomp}
\rmv(\wF,\wm):=\bigcup_{j \in \NN} \rmv(\wF, \wm, j) 
\end{equation}
(here, ``rm" stands for ``removable"). The diagram of $(\wF,\wm)$ introduced above is helpful to understand 
$\rmv(\wF,\wm)$. If there is a white square in the $j$-th column, 
the elements in $\rmv(\wF, \wm, j)$ correspond to the all squares in the $j$-th column.

We can take $r_1, \ldots, r_k$ so that 
$$\rmv(\wF, \wm)=\bigsqcup_{l=1}^k R_l \quad \text{with} \quad R_l := \rmv(\wF, \wm, j_{r_l}).$$ 
For $j:=j_{r_l}$, set $s:= \min \{ \, r \mid j_r=j \, \}$,  $t:= \max \{ \, r \mid j_r=j \, \}$ and 
$u:= \min \{ \, l > i_t \mid \text{$x_l$ divides $\m$}  \}$. Then,  we have  
\begin{equation}\label{eq:R_l decomp}
R_l=\{ \, (i_r,j) \in \wF \mid s \le r \le t \, \} \cup \{ \, (u, j) \, \}.
\end{equation}
In particular, we have $\# R_l \ge 2$, and $(i_r,j) \in \wF$ with $s \le r \le t$ (resp. $(u,j)$) 
corresponds to a white (resp. black) square in the $j$-th column of the diagram of $(\wF,\wm)$. 
For $r$ with $s \le r \le t$,  $r \in B(\wF,\wm)$ if and only if $r=t$ by Lemma~\ref{one generated}. 

\begin{lem}\label{rmv 1}
With the above notation (in particular, $j=j_{r_l}$), for $r$ with $s \le r \le t$, we have 
$$x(\wF_r, \wm)=x(\wF, \wm)/x_{i_r, j}$$
and $$\rmv(\wF_r, \wm)=\begin{cases}
\rmv(\wF, \wm) \setminus \{ x_{i_r, j}\} & \text{$\# R_l > 2$,}\\
\rmv(\wF, \wm) \setminus R_l  & \text{$\# R_l = 2$.}
\end{cases}.$$ 
Similarly, we have 
$$x(\wF_t, \wm_{\<i_t\>})=x(\wF, \wm)/x_{u, j}$$
and 
$$\rmv(\wF_t, \wm_{\<i_t\>})=\begin{cases}
\rmv(\wF, \wm) \setminus \{ x_{u, j}\} & \text{$\# R_l > 2$,}\\
\rmv(\wF, \wm) \setminus R_l  & \text{$\# R_l = 2$.}
\end{cases}.$$ 
\end{lem}

\begin{proof}
Easily follows from the above observation. 
\end{proof}

Let $A_I$ be the set of all admissible pairs for $I$. We make $A_I$ a poset 
(i.e., partially ordered set) as follows:
$(\wF, \wm)$ covers $(\wF', \wm')$ if and only if $\pm e(\wF',\wm')$ appears in $\partial(e(\wF,\wm))$, 
where $\partial$ is the differential of $\wP_\bullet$.
Extending this, we can define the order $<$ on $A_I$.
We henceforth write $(\wF, \wm) \gtrdot (\wF', \wm')$
to mean $(\wF, \wm)$ covers $(\wF', \wm')$.
By Lemma~\ref{rmv 1}, $(\wF, \wm) \gtrdot (\wF', \wm')$ if and only if $x(\wF', \wm')=x(\wF,\wm)/x_{i,j}$ for 
some $(i,j) \in \rmv(\wF,\wm)$. 
In fact, if $x(\wF', \wm')=x(\wF,\wm)/x_{i,j}$, then $(\wF', \wm')$ is of the form $(\wF_r, \wm)$ or 
$(\wF_r, \wm_{\<i_r\>})$ by \eqref{diagram recovers}. 

\begin{prop}\label{rmv 2}
With the above situation, let $(\wF, \wm)$ and $(\wF', \wm')$ be admissible pairs of $\wI$, and consider the decomposition 
$\rmv(\wF, \wm)=\bigsqcup_{l=1}^k R_l$ as in \eqref{rm decomp}. 
Then $(\wF, \wm) \ge (\wF', \wm')$ 
if and only if there is a subset $R \subset \rmv(\wF, \wm)$ 
such that $R_l \not \subset R$ for all $l$ and 
$$x(\wF', \wm')=x(\wF,\wm)/\prod_{(i,j) \in R}x_{i,j}.$$
\end{prop}
\begin{proof}
Suppose $(\wF, \wm) \ge (\wF', \wm')$.
This means there exists a descending chain of cover relations
starting from $(\wF,\wm)$ and going down to $(\wF',\wm')$.
Applying Lemma~\ref{rmv 1} to each cover relations,
we obtains a subset $R \subset \rmv(\wF, \wm)$ 
such that $R_l \not \subset R$ for all $l$ and
$$
x(\wF', \wm')=x(\wF,\wm)/\prod_{(i,j) \in R}x_{i,j}.
$$

Conversely, suppose there exists such a subset $R \subset \rmv(\wF, \wm)$.
We will show $(\wF, \wm) \ge (\wF', \wm')$ by induction on $\#R$.
The case $\#R = 0$ is clear, since $x(\wF, \wm) = x(\wF', \wm')$
implies $(\wF, \wm) = (\wF', \wm')$ as is stated above.

Assume $\#R \ge 1$. Take $l$ so that $R \cap R_l \ne \emptyset$, and use the description \eqref{eq:R_l decomp} of $R_l$.
If $R \cap R_l =\{ (u,j)\}$, then we use $(\wF_t, \wm_{\<i_t\>})$.  
Note that $x(\wF_t, \wm_{\<i_t\>})=x(\wF, \wm)/x_{u, j}$ and  $(\wF, \wm) \gtrdot (\wF_t, \wm_{\<i_t\>})$. 
Moreover,  $(\wF_t, \wm_{\<i_t\>})$, $(\wF', \wm')$  and $R \setminus \{(u,j)\}$ keep the 
assumption of the proposition. Hence we have  $(\wF_t, \wm_{\<i_t\>}) \geq (\wF', \wm')$ by the induction hypothesis, 
and we are done. 

 If $R \cap R_l \ne \{ (u,j)\}$, the  we have $(i_r, j) \in R \cap R_l$ 
for some $s \le r \le t$.  
Note that $x(\wF_r, \wm)=x(\wF, \wm)/x_{i_r, j}$ and  $(\wF, \wm) \gtrdot (\wF_r, \wm)$. 
Moreover,  $(\wF_r, \wm)$, $(\wF', \wm')$  and $R \setminus \{(i_r,j)\}$ keep the 
assumption of the proposition (in this case, that $R_l \not \subset R$ is crucial to see 
$R \setminus \{ (i_r,j) \} \subset \rmv(\wF_r, \wm)$). 
Hence we have  $(\wF_r, \wm) \geq (\wF', \wm')$ by the induction hypothesis, 
and we are done. 
\end{proof}

Let $(\wF,\wm)$ be an element of $A_I$, and $R_1, \ldots, R_k$ as in Proposition~\ref{rmv 2}. 
By the proposition, it is easy to see that  the order ideal 
$\{ \, (\wF', \wm') \mid  (\wF', \wm') \le  (\wF, \wm) \, \}$ of $A_I$ is isomorphic to 
the Cartesian product $$(2^{R_1} \setminus \{ \emptyset\}) \times (2^{R_2} \setminus \{ \emptyset \}) \times \cdots 
\times (2^{R_k} \setminus \{ \emptyset \} ).$$ 
  
Under the assumption that $I$ is generated in one degree, Nagel and Reiner \cite{NR} constructed a polytopal complex (hence a regular CW complex)
which supports  a minimal free resolution of $\wI$ (or $I$, $I^\sq$). See \cite[Theorem~3.13]{NR}.  
If we regard their polytopal complex as a poset in the natural way, then it is isomorphic to our $A_I$ by Proposition~\ref{rmv 2}. 
Hence we have the following. (Since their complex is a regular CW complex,
the choice of an incidence function does not cause a problem.) 

\begin{prop} 
Let $I$ be a Borel fixed ideal generated in one degree. 
Then Nagel-Reiner description of a minimal free resolution of $\wI$ coincides with our $\wP_\bullet$ 
(more precisely, the truncation $\wP_{\ge 1}$, since $\wP_\bullet$ is a resolution of $\wS/\wI$, not $\wI$ itself).  
\end{prop}

\begin{exmp}\label{exmp:tetra-pri}
Let us consider the ideal $I = (x_1^2, x_1x_2, x_1x_3, x_1x_4, x_2^2, x_2x_3, x_2x_4)$ of $\kk[x_1,x_2,x_3,x_4]$
as in Remark~\ref{rem:diff res}.
Then
$$
\wI = (x_{1,1}x_{1,2}, x_{1,1}x_{2,2},  x_{1,1}x_{3,2}, x_{1,1}x_{4,2}, x_{2,1}x_{2,2},
x_{2,1}x_{4,2}).
$$
We set $\wF := \{ (1,2),(2,2),(3,2) \}$ and $\wF' := \{ (1,1),(2,2),(3,2)\}$.
It is straightforward to verify that
$(\wF,x_{1,1}x_{4,2})$ and $(\wF', x_{2,1}x_{4,2})$ are the maximal admissible pairs in $A_I$.
The diagrams of these admissible pairs are as follows.

\begin{figure}[h]
\setlength\unitlength{.4mm}
\begin{minipage}{.3\textwidth}
\begin{center}
\begin{picture}(65,70)(0,0)
\thicklines
\put(20,40){\shade\path(0,0)(0,10)(10,10)(10,0)(0,0)}
\put(30,40){\whiten\path(0,0)(0,10)(10,10)(10,0)(0,0)}
\put(30,30){\whiten\path(0,0)(0,10)(10,10)(10,0)(0,0)}
\put(30,20){\whiten\path(0,0)(0,10)(10,10)(10,0)(0,0)}
\put(30,10){\shade\path(0,0)(0,10)(10,10)(10,0)(0,0)}
\put(5,30){\makebox(0,0){$i$}}
\put(30,70){\makebox(0,0){$j$}}
\put(10,13){$4$}
\put(10,23){$3$}
\put(10,33){$2$}
\put(10,43){$1$}
\put(23,56){$1$}
\put(33,56){$2$}
\end{picture}
\end{center}
\end{minipage}
\begin{minipage}{.3\textwidth}
\begin{center}
\begin{picture}(65,70)(0,0)
\thicklines
\put(20,40){\whiten\path(0,0)(0,10)(10,10)(10,0)(0,0)}
\put(20,30){\shade\path(0,0)(0,10)(10,10)(10,0)(0,0)}
\put(30,30){\whiten\path(0,0)(0,10)(10,10)(10,0)(0,0)}
\put(30,20){\whiten\path(0,0)(0,10)(10,10)(10,0)(0,0)}
\put(30,10){\shade\path(0,0)(0,10)(10,10)(10,0)(0,0)}
\put(5,30){\makebox(0,0){$i$}}
\put(30,70){\makebox(0,0){$j$}}
\put(10,13){$4$}
\put(10,23){$3$}
\put(10,33){$2$}
\put(10,43){$1$}
\put(23,56){$1$}
\put(33,56){$2$}
\end{picture}
\end{center}
\end{minipage}
\end{figure}

It follows that $x(\wF, x_{1,1}x_{4,2}) = x_{1,1}x_{4,2} \times x_{1,2}x_{2,2}x_{3,2}$ and
$x(\wF', x_{2,1}x_{4,2}) = x_{2,1}x_{4,2} \times x_{1,1}x_{2,2}x_{3,2}$.
The sets $\rmv(\wF, x_{1,1}x_{4,2})$ and $\rmv(\wF', x_{2,1}x_{4,2})$ are decomposed
as follows:
\begin{align*}
&\rmv(\wF, x_{1,1}x_{4,2}) = R_1 = \{ (1,2), (2,2), (3,2), (4,2) \}, \\
&\rmv(\wF', x_{2,1}x_{4,2}) = R_1 \sqcup R_2 = \{ (1,1), (2,1) \} \sqcup
\{ (2,2), (3,2), (4,2) \}.
\end{align*}
The order ideals generated by $(\wF,x_{1,1}x_{4,2})$
and $(\wF', x_{2,1}x_{4,2})$, respectively, are
\begin{align*}
&(2^{\{ (1,2),(2,2),(3,2),(4,2) \}} \setminus \{ \emptyset \}), \\
&(2^{\{ (1,1),(1,2) \}} \setminus \{ \emptyset \}) \times
(2^{\{ (2,2),(3,2),(4,2) \}} \setminus \{ \emptyset \}).
\end{align*}
The former corresponds to a tetrahedron, and the latter to a triangular prism;
indeed the polytopal complex supporting the minimal free resolution of $I$
as in Theorem~\ref{main} can be realized by gluing them as in the figure below.

\begin{figure}[htbp]
\begin{center}
\includegraphics[width = 7cm]{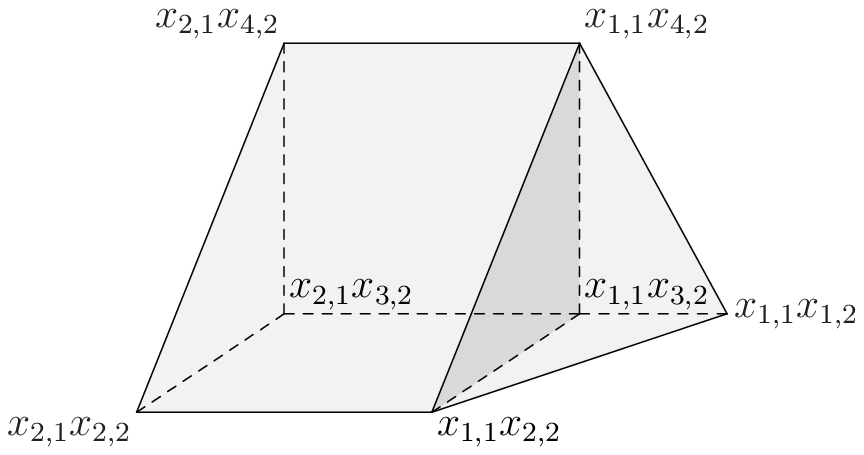}
\end{center}
\end{figure}
\end{exmp}

In the next section, we will show that our resolution $\wP_\bullet$ is cellular even if 
$I$ is {\it not} generated in one degree.  
 
\section{Relation to Batzies-Welker theory}
We use the same notation as in the previous sections. 
In particular, $I$ is a Borel fixed ideal, and $\wI :=\BoX(I)$ is the alternative polarization. 

In \cite{BW}, Batzies and Welker connected the theory of {\it cellular resolutions} of 
monomial ideals with Forman's discrete Morse theory (\cite{F}). 
In this section,  we will  show that our resolution $\wP_\bullet$ of $\wS/\wI$ 
can also be obtained by their construction.  

\begin{dfn}\label{shellable}
A monomial ideal $J$ is called {\it shellable} if there is a total order $\sqsubset$ on $G(J)$ 
satisfying the following condition.  
\begin{itemize}
\item[$(*)$] For any $\m, \m' \in G(J)$ with $\m \sqsupset \m'$, 
there is an $\m'' \in G(J)$ such that $\m \sqsupseteq \m''$, $\deg\left(\frac{\lcm(\m, \m'')}{\m}\right)=1$ 
and $\lcm(\m, \m'')$ divides $\lcm(\m, \m')$.   
\end{itemize}
\end{dfn}

\begin{rem}\label{lq}
One can show that $J$ is shellable in the above sense if and only if $I$ has linear quotients  
in the sense of \cite{HT}. 
\end{rem}
 
For a shellable monomial ideal, we have a Batzies-Welker type {\it minimal} free resolution. 
Since a Borel fixed ideal is shellable, we can apply their method. 
However, their description is not as explicit as that of Eliahou and Kervaire 
(see, for example, \cite[\S6.2]{JW}).       
It makes our argument in this section complicated. 

\medskip

Let $\sqsubset$ be  the total order on $G(\wI) =\{ \, \wm \mid \m \in G(I) \, \}$ such that $\wm' \sqsubset \wm$ 
if and only if $\m' \succ \m$ in the lexicographic order of $S$ with $x_1 \succ x_2 \succ \cdots \succ x_n$.
Hence the order $\sqsubset$ is just the opposite of the lexicographic order.
In the rest of this section, $\sqsubset$ means this order.

\begin{lem}\label{wI is shelllable}
The order $\sqsubset$ makes $\wI$ shellable. 
\end{lem}

\begin{proof}
If we set $l := \min \{ \, i \mid \text{$x_i$ divides $\lcm (\m, \m')/\m$} \, \}$,  
then $\wm''= \wm_{\< l \>}$ satisfies the expected property (see also the proof of Lemma~\ref{sec:bpolfilt}). 
\end{proof}

Hence we have a Batzies-Welker type minimal free resolution  of  $\wI$. 
This resolution is cellular, and the corresponding CW complex $X_A$ 
is obtained from the simplex $X$ (essentially, the power set $2^{G(\wI)}$, which supports 
the Taylor resolution of $\wI$) using discrete Morse theory. 

The following construction is taken from \cite[Theorems~3.2 and 4.3]{BW}.  
For the background of their theory, the reader is recommended to consult the original paper. 


\medskip

For $\emptyset \ne \gs \subset G(\wI)$,  let $\wm_\gs$  denote the largest element of $\gs$ with respect to 
the order $\sqsubset$, and set $$\lcm(\gs):= \lcm \{ \, \wm \mid \wm \in \gs \, \}.$$
For simplicity, set $(\wm_\gs)_{\<i\>} := \BoX((\m_\gs)_{\<i\>})$. This notation is confusing 
about the order of the operations $(-)_{\<i\>}$ and  $\BoX(-)$, but we could not find another concise symbol.

\begin{dfn}\label{order gs} 
We define a total order $\prec_{\gs}$ on $G(\wI)$ as follows. 
\begin{itemize} 
\item Set $N_\gs:= \{ \, (\wm_\gs)_{\<i\>} \mid \text{$1 \le i < \nu(\m_\gs)$,  
$(\wm_\gs)_{\<i\>}$ divides $\lcm(\gs)$} \, \}$. 
\item For all $\wm \in N_\gs$ and $\wm' \in G(\wI) \setminus N_\gs$, we have $\wm \prec_{\gs} \wm'$.
\item The restriction of $\prec_{\gs}$ to $N_\gs$ coincides with $\sqsubset$, and the same is true for 
the restriction to $G(\wI) \setminus N_\gs$. 
\end{itemize}
\end{dfn}

\begin{rem}
In the above definition,  $N_\gs$ plays the same role as  $\{\, n_j^m \mid j \in J_m\, \}$ in the proof of \cite[Proposition~4.3]{BW}. 
For a general shellable monomial ideal (e.g., a Borel fixed ideal),
there is an ambiguity in the choice of $\{\, n_j^m\mid j \in J_m \, \}$, 
but $N_\gs$ is the unique one in our case. In this sense, $\wI$ is simpler than $I$ itself.  
\end{rem}

Let $X$ be the $(\#G(\wI)-1)$-simplex associated with $2^{G(\wI)}$ (more precisely,  $2^{G(\wI)} \setminus \{\emptyset \}$). 
Hence we freely identify $\gs \subset G(\wI)$ with the corresponding cell of the simplex $X$. 
Let $G_X$ be the directed graph defined as follows. 
\begin{itemize}
\item The vertex set of $G_X$ is $2^{G(\wI)} \setminus \{ \emptyset \}$.  
\item For $\emptyset \ne \gs, \gs' \subset G(\wI)$, there is an arrow $\gs \to \gs'$ if and only if 
$\gs \supset \gs'$ and $\#\gs = \# \gs'+1$.  
\end{itemize}

For $\gs =\{ \, \wm_1, \wm_2, \ldots, \wm_k \, \}$ with $\wm_1 \prec_{\gs} \wm_2 \prec_{\gs}  
\cdots \prec_{\gs}  \wm_k \, (=\wm_\gs)$ and $l \in \NN$ with $1 \le l <k$,  
set $\gs_l:= \{ \, \wm_{k-l}, \wm_{k-l+1}, \ldots, \wm_k \, \}$ and 
$$u(\gs):= \sup \{ \,  l \mid \text{$\exists \wm \in G(I)$ s.t. $\wm \prec_{\gs} \wm_{k-l}$ and $\wm | \lcm(\gs_l)$ }\, \}.$$ 
If $u:=u(\gs) \ne -\infty$, we can define  
$$
\wn_\gs := \min\nolimits_{\prec_{\gs}} \{ \, \wm \mid \text{$\wm$ divides $\lcm(\gs_u)$} \, \}.
$$
Let $E_X$ be the set of edges of $G_X$. We define a subset $A$ of $E_X$ by 
$$
A:= \{ \, \gs  \cup \{ \wn_\gs \} \to \gs \mid u(\gs) \ne -\infty,  \wn_\gs \not \in \gs \, \}.
$$
Let $G_X^A$ be the directed graph with the vertex set $2^{G(\wI)} \setminus \{ \emptyset \}$ (i.e., same as $G_X$) and the set of edges 
$$
(E_X \setminus A) \cup \{ \, \gs \to \gt \mid (\gt \to \gs) \in A \, \}.
$$
In the proof of \cite[Theorem~3.2]{BW}, it is  shown that $A$ is a {\it matching}, i.e.,  
every $\gs$ occurs in at most one edges of $A$, and
moreover it is {\em acyclic}, i.e., $G_X^A$ has no directed cycle. 
We say $\emptyset \ne \gs \subset G(\wI)$ is {\it critical}, if it does not occur in any edge of $A$. 
A directed path  in $G_X^A$ is called a {\it gradient path}. 

\begin{exmp}\label{exmp:matching}
Let $I = (x^2, xy,xz,xw, y^2,yz, yw)$ be the ideal in Example~\ref{exmp:tetra-pri}.
(Here we set $x := x_1, y:=x_2, z:=x_3, w:=x_4$).
Then
$$
x_1x_2 \sqsubset x_1y_2 \sqsubset x_1z_2 \sqsubset x_1w_2 \sqsubset y_1y_2 \sqsubset y_1z_2 \sqsubset y_1w_2.
$$
Let $\gs := \{ x_1w_2, y_1y_2, y_1z_2 \}$. Then  $\wm_\gs = y_1z_2$, $\lcm(\gs) = x_1y_1y_2z_2w_2$,
and $N_\gs =  \{ x_1z_2, y_1y_2\}$.
Hence 
$$
x_1z_2 \prec_\gs y_1y_2 \prec_\gs x_1x_2 \prec_\gs x_1y_2 \prec_\gs x_1w_2 \prec_\gs y_1z_2 \prec_\gs y_1w_2.
$$
With respect to $\prec_\gs$, the elements $\gs$ are ordered as $\{ y_1y_2, x_1w_2, y_1z_2 \}$.
Since $x_1z_2 \prec_\gs y_1y_2$ and $x_1z_2$ divides $\lcm(\gs_2) = \lcm(y_1y_2, x_1w_2, y_1z_2)$,
we see that $u(\gs) = 2$ and $\wn_\gs = x_1z_2$.
Thus $\{ x_1z_2, x_1w_2, y_1y_2, y_1z_2 \} = \gs \cup \{ \wn_\gs \} \to \gs$ belongs to the matching $A$ for $I$.

It is a routine to verify that there exists the following gradient path
\begingroup
\xymatrixcolsep{-4pc}
$$
\xymatrix{
\{ x_1w_2, y_1y_2, y_1z_2, y_1w_2\}   \ar[dr] \ar@{-->}[drrrrr]& & \{ x_1z_2,x_1w_2, y_1y_2, y_1z_2 \}  \ar[dr]&
& \{ x_1y_2, x_1z_2, x_1w_2 , y_1y_2 \} \ar[dr] & \\
& \{ x_1w_2, y_1y_2, y_1z_2 \}  \ar@{=>}[ur]&                             &
\{ x_1z_2, x_1w_2, y_1y_2 \}  \ar@{=>}[ur]&                               & \{ x_1y_2, x_1z_2, x_1w_2 \}
}
$$
\endgroup
where double arrows denotes reversed ones (hence their original arrows belong to the matching $A$).
Note that $\{ x_1w_2, y_1y_2, y_1z_2, y_1w_2\}$ and $\{ x_1y_2, x_1z_2, x_1w_2 \}$
are critical (see \eqref{shape of sigma}).
\end{exmp}

Forman's discrete Morse theory \cite{F} (see also \cite{C}) guarantees the existence of a CW complex $X_A$ with the following conditions.
\begin{itemize}   
\item There is a one-to-one correspondence between the $i$-cells of $X_A$ and the {\it critical} $i$-cells of $X$. 
\item $X_A$ is contractible, that is, homotopy equivalent to $X$.  
\end{itemize}
The cell of $X_A$ corresponding to a critical cell $\gs$ of $X$ is denoted by $\gs_A$. By \cite[Proposition~7.3]{BW}, 
the closure of $\gs_A$ contains $\gt_A$ if and only if there is a gradient path from $\gs$ to $\gt$. 
See also Proposition~\ref{gradient path} below  and the argument before it. 

The following is a direct consequence of \cite[Theorem~4.3]{BW}. 

\begin{prop}[{Batzies-Welker, \cite{BW}}] 
With the above notation, the CW complex $X_A$ supports a minimal free resolution $\cF_{X_A}$ of $\wI$.  
\end{prop}

While the explicit form of $\cF_{X_A}$ will be introduced before Theorem~\ref{BW=wP}, we remark now that 
the free summands of $\cF_{X_A}$ are indexed by the cells of $X_A$ (equivalently, the critical subsets of $G(\wI)$). 

The purpose of this section is to show  that the Batzies-Welker resolution $\cF_{X_A}$ is  ``same" as our $\wP_\bullet$, 
that is, there is the isomorphism $\cF_{X_A} \cong \wP_\bullet$ given by a correspondence between  
the basis elements of $\cF_{X_A}$ and those of $\wP_\bullet$. 
For consistent notation, set $\wQ_\bullet:=\cF_{X_A}$. 
The following (long) discussion is necessary to compute the differential map of  $\wQ_\bullet$.   

\begin{lem}\label{lcm lemma}
Let $\gs$ and  $\gt$ be (not necessarily critical) subsets of  $G(\wI)$ admitting a gradient path from $\gs$ to $\gt$. 
Then $\lcm(\gt)$ divides $\lcm(\gs)$. 
\end{lem}

\begin{proof}
It suffices to consider the case there is an arrow $\gs \to \gt$ in $G_X^A$.  
If $\gs \supset \gt$, the assertion is  clear. 
If  $\gt \supset \gs$, then $\gt=\gs \cup \{ \wn_\gs\}$ and the assertion follows from the fact that $\wn_\gs$ divides $\lcm(\gs)$.  
\end{proof}

Assume that  $\emptyset \ne \gs \subset G(\wI)$ is critical. 
Recall that $\wm_\gs$ denotes the largest element of $\gs$ with respect to $\sqsubset$.   
Take $\m_\gs = \prod_{l=1}^n x_l^{a_l} \in G(I)$  with $\wm_\gs=\BoX(\m_\gs)$, and set $q := \#\gs-1$. 
Then there are integers $i_1, \ldots , i_q$ with $1 \le i_1 < \ldots < i_q < \nu(\m_\gs)$ and 
\begin{equation}\label{shape of sigma}
\gs=\{ \, (\wm_\gs)_{\<i_r\>} \mid 1 \le r \le q \, \} \cup \{ \wm_\gs   \}
\end{equation}
(see the proof of \cite[Proposition~4.3]{BW}). 
Equivalently, we have $\gs =N_\gs \cup \{  \wm_\gs  \}$ in the notation of Definition~\ref{order gs}. 
Set $j_r:= 1+ \sum_{l=1}^{i_r}a_l$ for each $1 \le r \le q$, and 
$\wF_\gs :=\{ \, (i_1, j_1), \ldots, (i_q, j_q) \, \}$.  
Then $(\wF_\gs, \wm_\gs)$ is an admissible pair for $\wI$. 
Conversely, any admissible pair comes from a critical cell $\gs \subset G(\wI)$ in this way. 
Hence there is a one-to-one correspondence between critical cells and admissible pairs. 
Note that 
\begin{equation}\label{lcm of sigma}
\lcm (\gs)= \wm_\gs \times \prod_{r=1}^qx_{i_r, j_r}
\end{equation}
is the ``support" of the diagram of the admissible pair $(\wF_\gs, \wm_\gs)$ introduced in \S4. 

From now on,  we study gradient paths starting from a critical subset $\gs$ of the form \eqref{shape of sigma}. 
Since $\gs$ is critical, the first step of any path must be $\gs \to \gs \setminus \{\wm\}$ for some $\wm \in \gs$. 
Let $\gt \subset G(\wI)$  be a critical subset with $\#\gt=\#\gs-1=q$, and assume that there is a gradient path 
$\gs \to \gs \setminus \{\wm\} =\gs_0 \to \gs_1 \to \cdots \to \gs_l=\gt$. 
Since $\gt$ is critical, we have $\# \gs_{l-1} = \#\gt+1 = q+1$.  
It follows from $A$ being a matching that $\# \gs_i=q$ or $q+1$ for each $i$,
and $\gs_i$ is not critical for all $0 \le i < l$. 
Since  $\gs \setminus \{ \, (\wm_\gs)_{\<i_r\>} \, \}$ is critical for $1 \le r \le q$, 
it remains to consider a gradient path starting from $\gs \to \gs \setminus \{ \wm_\gs \}$.

\begin{prop}\label{gradient path}
Let $\gs, \gt \subset G(\wI)$ be critical subsets with $\# \gs=\#\gt+1$, 
and $(\wF_\gs, \wm_\gs)$ and $(\wF_\gt, \wm_\gt)$ 
the admissible pairs corresponding to $\gs$ and $\gt$ respectively. 
Set $\wF_\gs=\{ \, (i_1, j_1), \ldots, (i_q, j_q) \, \}$ with $i_1 <  \cdots < i_q$. 
There is a gradient path $(\gs \setminus \{ \wm_\gs \})
=\gs_0 \to \gs_1 \to  \cdots \to \gs_l =\gt$  
if and only if there is some $r \in B(\wF_\gs, \wm_\gs)$ with 
$(\wF_\gt, \wm_\gt)=((\wF_\gs)_r, (\wm_\gs)_{\<i_r\>})$. 
\end{prop}

\begin{proof}
First, we  assume that there is a gradient path from $\gs \setminus \wm_\gs$ to $\gt$. 
By Lemma~\ref{lcm lemma} and  \eqref{lcm of sigma}, 
the support of diagram of $(\wF_\gt, \wm_\gt)$ is a ``subset" of that of $(\wF_\gs, \wm_\gs)$. 
By the definition of gradient paths, we have 
$\wm \sqsubseteq \max_{\prec_\gs} (\gs \setminus \{ \wm_\gs \} ) \sqsubset \wm_\gs$ for all $\wm \in \gt$. 
It follows that $\wm_\gt \ne \wm_\gs$, and the set of the black squares in the diagram of $(\wF_\gs, \wm_\gs)$ 
is different from that of $(\wF_\gt, \wm_\gt)$. 
By the assumption $\# \gs = \# \gt+1$ (equivalently, $\#\wF_\gs =\#\wF_\gt+1$), 
the number of the white square in the diagram of $(\wF_\gt, \wm_\gt)$ is one smaller than that of $(\wF_\gs, \wm_\gs)$. 
By the shapes of the diagrams of admissible pairs,  we have 
$(\wF_\gt, \wm_\gt)=((\wF_\gs)_r, (\wm_{\gs})_{\< i_r \>} )$ for some $r \in B(\wF, \wm_\gs)$.  

Next, assuming $\wF_\gt=(\wF_\gs)_r$ and $\wm_\gt =(\wm_\gs)_{\< i_r\>}$ for some $r \in B(\wF_\gs, \wm_\gs)$, 
we will construct a gradient path from $\gs \setminus \{ \wm_\gs \}$ to $\gt$. 
For short notation, set 
$$\wm_{[s]}:= (\wm_\gs)_{\<i_s\>} \qquad \text{and} \qquad \wm_{[s,t]}:= ((\wm_\gs)_{\<i_s\>})_{\<i_t\>}.$$
By \eqref{shape of sigma}, we have $\gs_0 := (\gs \setminus \{ \wm_\gs \}) =\{ \, \wm_{[s]} \mid 1 \le s\le q \, \}$ and  
\begin{eqnarray*}
\tau &=& \{ \, (\wm_\gt)_{\< i_s \>} \mid 1 \le s \le q, s \ne r\, \} \cup \{ \wm_\gt \}\\
&=& \{ \, \wm_{[r,s]} \mid 1 \le s \le q, s \ne r\, \} \cup \{ \wm_{[r]} \}.
\end{eqnarray*}

\smallskip

\noindent{\it Case 1.} 
First consider the case $r=q$. If $q=1$, then $\gs_0=\gt$. So we may assume that $q \ge 2$. 
Then $\wm_\gt= \wm_{[q]}$ is the largest element of $\sigma_0$ with respect to 
the order $\sqsubset$. (The same is true for $\gs_1, \ldots, \gs_{2(r-1)}$ below.) 

We might have  $\wm_{[s]} = \wm_{[q,s]}$ for some $s < q$. 
For instance, if $j_1=j_q$, then the equality holds for all $s$.  
Even if $j_s < j_q$ we can not ignore the effect of $g(-)$. 
However, we always have $\wm_{[q,s]} | \lcm \{ \wm_{[s]}, \wm_{[q]}\}$. 
 
First, we assume that $\wm_{[s]} \ne \wm_{[q,s]}$ for all $s < q$.  
Note that $\wm_{[q,1]}$ divides $\wn :=\lcm(\gs_0)$, and it is the smallest element in $G(\wI)$ 
with respect to the order $\prec_{\gs_0}$ defined above. 
By the present assumption that $\wm_{[1]} \ne \wm_{[q,1]}$, we have $\wm_{[q,1]} \not \in \gs_0$.  
Hence if we set $\gs_1:= \gs_0 \cup \{ \wm_{[q,1]} \}$, then $\gs_0 \to \gs_1$ is an arrow in $G_X^A$. 
Clearly,  we have an arrow $\gs_1 \to \gs_2$, where $\gs_2:= \gs_1 \setminus \{\wm_{[1]}\}$.   
If $q=2$, then $\gs_2=\gt$, and we are done. So assume that $q >2$. 
While $x_{i_1, j_1}$ divides $\wm_{[q,1]}$, it does not 
divide $\wn':=\lcm ( \gs_2  \setminus \{ \wm_{[q,1]}\})$. 
Hence $\wm_{[q,1]}$ does not divide $\wn'$, 
and $\wm_{[q,2]}$ is the smallest elements of $\{ \, \wm \in G(\wI) \mid \text{$\wm$ divides $\wn'$ } \}$  
with respect to $\prec_{\gs_2}$. If we set $\gs_3 := \gs_2 \cup \{\wm_{[q,2]}\}$ 
and $\gs_4 := \gs_3 \setminus \{ \wm_{[2]} \}$, then  $\gs_2 \to \gs_3 \to \gs_4$ is a gradient path. 
If $q=3$, then $\gs_4=\gt$. If $q>3$, we repeat this procedure till we get $\gs_{2(q-1)}$. Then 
$$\gs_{2(q-1)}=(\gs_0 \setminus \{ \, \wm_{[1]}, \ldots \wm_{[q-1]}\, \}) \cup \{ \, \wm_{[q,1]}, \ldots \wm_{[q,q-1]}\, \}=\gt,$$
and we are done. 

Next we consider the case $\wm_{[s]}=\wm_{[q,s]}$ for some $s < q$. 
Then, in the above construction, we skip the part 
$\gs_k \to \gs_k \cup \{ \wm_{[q,s]}\} \to (\gs_k \cup \{ \wm_{[q,s]}\} ) \setminus \{\wm_{[s]}\}$  
(these are ``loops" now), and just change its ``name" from $\wm_{[s]}$ to $\wm_{[q,s]}$.

\smallskip

\noindent{\it Case 2, Step 1.} Assume that $r <q$. We will use monomials $\wm_{[s,t]}$ for 
$s,t$ with $t \le r < s$. 
Clearly, $x_{i_t,j_t}$ divides  $\wm_{[s,t]}$. 
However,  by the effect of $g(-)$, $x_{i_s,j_s}$ need  not divide $\wm_{[s,t]}$, 
and we might have $\wm_{[q,t]} = \wm_{[t]}$ and $\wm_{[s, t]}=\wm_{[s+1,t]}$.   

For the simplicity, we assume that $\wm_{[q,t]} \ne \wm_{[t]}$ and 
$\wm_{[s, t]} \ne \wm_{[s+1,t]}$ for all  $t \le r < s$ in the following argument. 
If we have $\wm_{[q,t]} = \wm_{[t]}$ or $\wm_{[s, t]}=\wm_{[s+1,t]}$, we can avoid the problem as in Case~1.  
Since $r \in B(\wF_\gs,\wm_\gs)$, we have $\nu(\wm_{[r]}) > i_q$ and $x_{i_s,j_s} | \, \wm_{[r,s]}$ for all $r < s \le q$. 
Hence  $\wm_{[r,s]}$ is irrelevant to the above problem. 

We inductively define $\gs_1, \ldots, \gs_{2r-1}$  by 
$$\gs_{2s-1}:= \gs_{2s-2} \cup \{ \, \wm_{[q,s]} \, \}
=\{ \, \wm_{[q,1]}, \wm_{[q,2]}, \ldots , \wm_{[q,s]}, \wm_{[s]}, \wm_{[s+1]}, \ldots, \wm_{[q]}\, \}$$ 
for each $1 \le s \le r$, and   
$\gs_{2s}:= \gs_{2s-1} \setminus \{ \wm_{[s]}\}$ 
for each $1 \le s < r$. These are same as those in Case 1,  and we have a gradient path 
$\gs_0 \to \cdots \to \gs_{2r-1}$.  
Next we set 
$$\gs_{2r}:=\gs_{2r-1} \setminus \{\, \wm_{[q]} \, \}\\
= \{ \, \wm_{[q,1]}, \wm_{[q,2]}, \ldots , \wm_{[q,r]}, \wm_{[r]}, \wm_{[r+1]}, \ldots, \wm_{[q-1]}\, \}.$$
Of course, there is an arrow $\gs_{2r-1} \to \gs_{2r}$. If $r=q-1$, we move to Step 2 below. 
Hence we assume that $r < q-1$.  

Now $\wm_{[q-1]}$ is the largest element of $\sigma_{2r}$ with respect to $\sqsubset$. 
We inductively define $\gs_{2r+1}, \ldots, \gs_{4r-1}$ by 
\begin{eqnarray*}
\gs_{2r+2s-1}: &=& \gs_{2r+2s-2} \cup \{ \, \wm_{[q-1,s]} \, \}\\
&=&\{ \, \wm_{[q-1,t]} \mid 1 \le t \le s \, \} \cup \{ \, \wm_{[q,t]} \mid s \le t \le r \, \} \cup  \{\, \wm_{[t]} \mid r \le t \le q-1 \, \}.
\end{eqnarray*}  for $1 \le s \le r$,  and 
$\gs_{2r+2s}:= \gs_{2r+2s-1} \setminus \{ \wm_{[q, s]}\}$ for $1 \le s < r$. 
Then we have a gradient path $\gs_{2r} \to \gs_{2r+1} \to \cdots \to \gs_{4r-1}$ as before.  
Set 
\begin{eqnarray*}
\gs_{4r}&:=& \gs_{4r-1} \setminus \{\, \wm_{[q-1]} \, \}\\
        &=& 
\{ \, \wm_{[q-1,1]}, \wm_{[q-1,2]}, \ldots , \wm_{[q-1,r]}, \wm_{[q,r]}, \wm_{[r]}, \wm_{[r+1]}, \ldots, \wm_{[q-2]}\, \}.
\end{eqnarray*}
If $r=q-2$, we move to Step 2. If $r<q-2$, then we set $\gs_{4r+1}:= \gs_{4r} \cup \{ \, \wm_{[q-2,1]} \, \}$, 
$\gs_{4r+2}:= \gs_{4r+1} \setminus \{ \, \wm_{[q-1,1]} \, \}$, $\ldots, \gs_{6r}:= \gs_{6r-1} \setminus \{\wm_{[q-2]}\}$. 
We repeat this procedure till we get $\gs_{2r(q-r)}$, which equals 
$$\{ \, \wm_{[r+1,s]} \mid 1 \le s < r \, \} \cup  \{ \, \wm_{[r]} \, \} \cup 
\{ \, \wm_{[r, s]} \mid r  < s \le q \, \}$$
(since $r \in B(\wF_\gs,\wm_\gs)$, we have $\wm_{[s,r]}=\wm_{[r,s]}$ for all $s > r$ 
by Lemma~\ref{lem for adm} (iv)). Now we move to Step 2.

\smallskip

\noindent{\it Case 2, Step 2.}
Set $\gs':=\gs_{2r(q-r)}$ for short. The construction in this step is similar to Case 1. 
While we might have $\wm_{[s]}=\wm_{[r,s]}$ for some $s < r$, we can avoid the problem 
by the same way as Case 1. So we assume that  $\wm_{[s]} \ne \wm_{[r,s]}$ for all $s < r$.

Anyway, $\wm_{[r]}$ is the largest element of $\sigma'$ with respect to 
$\sqsubset$. By the same argument as in Step~1, we have  
a gradient path $\gs' =\gs'_0 \to \cdots \to \gs'_{2(r-1)}$ with $\gs'_{2s-1}=\gs'_{2s-2} \cup \{ \wm_{[r, s]}  \}$ 
and $\gs'_{2s}=\gs'_{2s-1} \setminus \{  \wm_{[r+1,s]}  \}$ for $1 \le s \le r-1$. We have  
$$\gs'_{2(r-1)}= (\gs' \setminus \{ \, \wm_{[r+1,s]} \mid 1 \le s < r \, \})   
\cup \{ \, \wm_{[r, s]} \mid 1  \le  s < r \, \} \\
=\gt.$$
\end{proof}

We also need the following. 

\begin{prop}\label{uniqueness}
Let $\gs, \gt \subset G(\wI)$ be as in Proposition~\ref{gradient path}.  Then the gradient path constructed in the proof of 
Proposition~\ref{gradient path} is the unique one connecting $\gs \setminus \{ \wm_\gs \}$ with $\gt$.  
\end{prop}

For the proof of the proposition, the next lemmas are useful. 

\begin{lem}\label{element of gs_k}
Let $\gs \subset G(I)$ be as in Proposition~\ref{gradient path}, and 
$(\gs \setminus \{\wm_\gs\}) =\gs'_0 \to \gs'_1 \to \cdots \to \gs'_k \to \cdots$ be a gradient path.  
Then, for any $\wm \in \gs'_k$, there is  a subset $\emptyset \ne 
\{ \, c_1, \ldots, c_s \, \} \subset \{ \, i_1, \ldots, i_q \, \}$ such that 
\begin{equation}\label{repeated}
\wm= (\,(\cdots (\, (\wm_{\gs})_{\<c_1\>})_{\<c_2\>} \cdots )_{\<c_{s-1}\>})_{\<c_s\>}.
\end{equation}
 \end{lem}
 
 \begin{proof}
Recall that $\lcm (\gs)$ coincides with the ``support" of the diagram of $(\wF_\gs, \wm_\gs)$, and 
$\lcm(\gs'_k)$ divides $\lcm(\gs)$. Hence $\wm \in \gs'_k$ divides the support of the diagram of $(\wF_\gs, \wm_\gs)$.  
By the properties of the diagrams, we see that any $\wm \in \gs'_k$ is either 
$\wm_\gs$ itself or of the form \eqref{repeated}. 
Finally, we  prove that $\wm \ne \wm_\gs$ by induction on $k$ using the fact 
that $\wm_\gs$ is larger than any  element of the form \eqref{repeated} with respect to the order $\sqsupset$.  
We leave the details to the reader as an easy exercise. 
\end{proof}

\begin{lem}\label{path lemma}
Let $\gs, \gt \subset G(I)$ be as in Proposition~\ref{gradient path},  
$(\gs \setminus \{\wm_\gs\}) =\gs'_0 \to \gs'_1 \to \cdots \to \gs'_{l'}= \gt$ any gradient path, 
and $r \in B(\wF_\gs, \wm_\gs)$ the one given in the proposition. 
Then $(\wm_\gs)_{\<i_r\>} \in \gs'_k$ for all $k$. 
\end{lem}

\begin{proof}
We use contradiction. Since $(\wm_\gs)_{\<i_r\>} \in \gt$, it suffices to show that $(\wm_\gs)_{\<i_r\>} \not \in \gs'_k$ implies 
$(\wm_\gs)_{\<i_r\>} \not \in \gs'_{k+1}$.  Assume that $(\wm_\gs)_{\<i_r\>} \not \in \gs'_k$. 
Since $r \in B(\wF_\gs, \wm_\gs)$,  we have $\wn_{\< c \>} \ne (\wm_{\gs})_{\<i_r\>}$ for all 
$\wn \in \gs_k'$ and all $c$ by Lemma~\ref{element of gs_k}. Similarly, if $x_{i_r, j_r}$ divides 
some $\wn \in \gs_k'$ (note that $x_{i_r, j_r}| (\wm_\gs)_{\<i_r\>}$), then $(\wm_\gs)_{\<i_r\>} \sqsupset \wn$. 
Hence the expected assertion follows from the construction of the acyclic matching $A$. 
\end{proof}

\begingroup
\renewcommand{\proofname}{The proof of Proposition~\ref{uniqueness}.}
\begin{proof}
Let $\cP:  (\gs \setminus \{ \wm_\gs \}) =\gs_0 \to \gs_1 \to  \cdots \to \gs_l =\gt$ be the path  
constructed in the proof of Proposition~\ref{gradient path}.  
We will show that if we once leave from $\cP$ then we can not arrive $\gt$ anymore. 
Since $G_X^A$ is acyclic, there is no (non-trivial) path from $\gt$ to $\gt$ itself. 
Hence it completes the proof of the uniqueness. 
Recall that $\# \gs_{2s-1} = \#\gs_{2s}+1=q+1$ for all $s$. 
If we take an arrow $\gs_{2s} \to \gs'$ with $\gs' \ne \gs_{2s-1}$, then we have $\# \gs' =q-1$, and  
there is no gradient path from $\gs'$ to $\gt$. 
Hence it suffices to consider the step $\gs_{2s-1} \to \gs'$. 
In this case, $\gs' = \gs_{2s-1} \setminus \{ \wm \}$ for some $\wm \in \gs_{2s-1}$. 
We follow the framework of the proof of Proposition~\ref{gradient path}
and also use the notation $\wm_{[s]} := (\wm_\gs)_{\<i_s \>}$ and
$\wm_{[s,t]} := ((\wm_\gs)_{\< i_s \>})_{\<i_t \>}$.

\smallskip

\noindent{\it Case~1.} First consider the case $r=q$. For the simplicity, we assume that $\wm_{[q,s]} \ne \wm_{[s]}$ 
for all $s < q$. The following argument also works without this assumption after minor modification. 

Recall that 
$$\gs_{2s-1}=\{\, \wm_{[q,1]}, \wm_{[q,2]}, \ldots, \wm_{[q,s]}, \wm_{[s]}, \wm_{[s+1]}, \ldots, \wm_{[q]} \, \}$$
and $\gs_{2s} =\gs_{2s-1} \setminus \{  \wm_{[s]} \}$ for each $1 \le s < q$. 
Since $\gs_{2s-1}=\gs_{2s-2} \cup \{ \wm_{[q,s]} \}$, we can not remove $\wm_{[q,s]}$ from $\gs_{2s-1}$. 
Hence it suffices  to show that there is no gradient path from $\gs':=\gs_{2s-1} \setminus \{  \wm  \}$ to $\gt$ 
for all $\wm \in \gs_{2s-1} \setminus \{ \, \wm_{[s]}, \wm_{[q,s]} \, \}$. 
If $\wm=\wm_{[q,t]}$ for some $1 \le t <s$, then $x_{i_t, j_t} \nmid\lcm(\gs')$, and hence 
$\lcm(\gt)  \nmid \lcm(\gs')$. By Lemma~\ref{lcm lemma},  there is no path from $\gs'$ to $\gt$. 
If $\wm=\wm_{[t]}$ for $t >s$, the same argument also works. 

\smallskip

\noindent{\it Case 2, Step 1.} Next consider the case  $r <q$.  
For the simplicity, we assume that $\wm_{[q,t]} \ne \wm_{[t]}$ and 
$\wm_{[s, t]} \ne \wm_{[s+1,t]}$ for all  $t \le r < s$.   
The following argument also works without this assumption after minor modification. 

Recall that the first $(2r-1)$-steps $\gs_0 \to \cdots \to \gs_{2r-1}$ of $\cP$ coincide with those in Case~1. 
In this part of the path, the proof in Case~1 also works. So we consider the next step $\gs_{2r-1} \to\gs'$. 
Note that 
$$\gs_{2r-1}= \{ \, \wm_{[q,1]}, \wm_{[q,2]}, \ldots , \wm_{[q,r]}, \wm_{[r]}, \wm_{[r+1]}, \ldots, \wm_{[q]}\, \}$$
and $\gs_{2r}=\gs_{2r-1} \setminus \{  \wm_{[q]}  \}$. 
We can not remove $\wm_{[q,r]}$ from $\gs_{2r-1}$, since it is the new comer of this set. 
If $\gs' = \gs \setminus \{  \wm  \}$ with $\wm=\wm_{[q,s]}$ for some $s <r$ or  $\wm=\wm_{[s]}$ for some $s >r$, 
then $x_{i_s, j_s}  \nmid \lcm(\gs')$, and hence $\lcm(\gt)  \nmid \lcm(\gs')$. 
So we must have $\gs' = \gs_{2r-1} \setminus \{ \wm_{[r]}  \}$, but this can not happen by Lemma~\ref{path lemma}. 

A similar argument works for the step $\gs_{2lr-1} \to \gs_{2lr}$ for $2 \le l \le q-r$. 
Note that 
$$\gs_{2lr-1}= \{ \, \wm_{[s,1]}, \wm_{[s,2]}, \ldots , \wm_{[s,r]}, \wm_{[r]}, \wm_{[r+1]}, \ldots, \wm_{[s]}, 
\wm_{[r,s+1]}, \wm_{[r,s+2]}, \ldots, \wm_{[r,q]} \, \}$$
and $\gs_{2lr}=\gs \setminus \{ \wm_{[s]}  \}$, where $s=q-l+1$.
Since  $\wm_{[s,r]}$ is the new comer, we can not remove it.  
If $\gs'= \gs_{2lr-1} \setminus \{  \wm_{[s,t]}  \}$ for $1 \le t <r$, there is no gradient path from 
$\gs'$ to $\gt$. In fact, the variable $x_{i_t, j_t}$ matters. 
The same is true for $\wm_{[t]}$ for $r < t \le s$ and $\wm_{[r, t]}$ for $s < t \le q$
(since $r \in B(\wF,\wm)$, $x_{i_t,j_t}$ divides $\wm_{[r, t]}$).  
Hence we may assume that  $\gs' = \gs_{2lr-1} \setminus \{  \wm_{[r]} \}$, but the assertion follows from Lemma~\ref{path lemma} in this case.

The remaining case of this step is $\gs_{2(lr+t)-1} \to \gs'$  for $2 \le l < q-r$ and $1 \le t <r$. 
Note that 
$$\gs_{2(lr+t)-1} = \{ \, \wm_{[s,u]} \mid 1 \le u \le t \, \} \cup\{ \, \wm_{[s+1,u]} \mid t \le u < r \, \}$$
$$\qquad \qquad \qquad \qquad \qquad \cup \{ \, \wm_{[u]} \mid r \le u \le s \, \} \cup  \{ \, \wm_{[r,u]} \mid s+1 \le u \le q \, \},$$
$\gs_{2(lr+t)} = \gs_{2(lr+t)-1}  \setminus \{  \wm_{[s+1,t]} \}$, and $\wm_{[s,t]}$ is the new comer of $\gs_{2(lr+t)-1}$. 
For $\wm_{[r]}$, we can use Lemma~\ref{path lemma} again. 
For the other elements except $\wm_{[s]}$ and $\wm_{[r,s+1]}$, the variable $x_{i_u,j_u}$ matters ($u \ne r, s, s+1$ in this case). 

To see that there is no gradient path from $\gs'_0 := \gs_{2(lr+t)-1} \setminus \{  \wm_{[s]} \}$ to $\gt$, we will show that 
if there is a gradient path $\gs'_0 \to \cdots \to \gs'_k$, then $\wm_{[r,s]} \not \in \gs'_k$. In fact, we can prove more:  
$\gs_k'$ does not contain any $\wn \in G(\wI)$ such that $\wn \sqsupseteq \wm_{[r,s]}$ and $x_{i_s, j_s} \, | \, \wn$. 
For the contradiction, assume that $\gs'_k$ contains such $\wn$ and $\gs'_{k-1}$ does not.  
Then we must have 
\begin{equation}\label{impossible}
\gs_k'= \gs'_{k-1} \cup \{ \wn\},  \ \ \ \  (\gs_k'\to  \gs'_{k-1}) \in A \ \ \ \  \text{and}  \ \ \ \   
\wn = (\wm_{\gs'_{k-1}})_{\<i_s\>}.
\end{equation} 
Since $x_{i_s, j_s} \, | \, \lcm (\gs'_{k-1})$, there is some $\wn' \in \gs'_{k-1}$ with $x_{i_s, j_s} \, | \, \wn'$. 
By the assumption, we have $\wn' \sqsubset \wm_{[r,s]}$ and there is some $u <r \, (<s)$ with $x_{i_u, j_u} \, | \, \wn'$. 
Since $(\wm_{\gs'_{k-1}})_{\<i_u\>} \ne \wn'$ (to see this, consider the variable $x_{i_s,j_s}$), 
the situation \eqref{impossible} can not happen by the construction of the matching $A$. 

We can prove that there is no gradient path from $\gs_{2(lr+t)-1} \setminus \{  \wm_{[r,s+1]} \}$ to $\gt$
by a similar argument.  

\smallskip

\noindent{\it Case 2, Step 2.} Now we connect $\gs_{2(q-r)}$ with $\gt$. Since the construction is quite similar to Case~1, 
we can prove the assertion by a similar way to that case.  
\end{proof}
\endgroup

Applying the construction of Batzies and Welker (see \cite[Remark~4.4]{BW} for its most explicit form), 
we get the following resolution $\wQ_\bullet$ of $\wI=\BoX(I)$. 
For a critical cell $\gs \subset G(\wI)$, 
$e(\gs)$ denotes a basis element with degree $\deg(\lcm(\gs)) \in \ZZ^{n \times d}$.   
Set 
$$\wQ_q = \bigoplus_{\substack{\text{$\gs:$ critical}\\ \#\gs=q+1}}\wS \, e(\gs) \qquad (q \ge 0).$$
For a critical cell $\gs$ of the form \eqref{shape of sigma},  the differential map sends $e(\gs)$ to  
\begin{eqnarray}\label{cell diff}
\sum_{r=1}^q &(-1)^r& x_{i_r, j_r} \cdot e(\gs \setminus \{ (\wm_\gs)_{\< i_r \>} \})\\ 
&-& (-1)^q \sum_{\substack{\text{$\gt:$ critical, $\# \gt=\#\gs-1$} \\ \text{$\exists$ grad.\,path $\cP: (\gs \setminus \{ \wm_\gs \})
 \leadsto \gt$}}}m(\cP) \cdot 
\frac{\lcm(\gs)}{\lcm(\gt)} \cdot e(\gt) \nonumber 
\end{eqnarray}
(Recall that a critical cell is characterized as a cell of the form \eqref{shape of sigma}.
Clearly $\gs \setminus \{ (\wm_\gs)_{\< i_r \>} \}$ is of such form).
The definition of the integer coefficient $m(\cP)$ can be found in the proof of Theorem~\ref{BW=wP} below (the original definition is in 
\cite[p.166]{BW}). 


\begin{thm}\label{BW=wP}
Our description of the resolution $\wP_\bullet$ coincides with the Batzies-Welker type resolution $\wQ_\bullet$ 
(more precisely, the truncation $\wP_{\ge 1}$ of $\wP_\bullet$ coincides with $\wQ_\bullet$).
\end{thm}

\begin{proof}
Recall that there is the one-to-one correspondence between the critical cells $\gs \subset G(\wI)$ and 
the admissible pairs $(\wF_\gs, \wm_\gs)$. Hence, for each $q$, we have the isomorphism $\wQ_q \to \wP_q$ induced 
by $e(\gs) \longmapsto e(\wF_\gs, \wm_\gs)$. We will show that they induce a chain isomorphism $\wQ_\bullet \cong \wP_\bullet$.

Note that $\gs \setminus  \{ (\wm_\gs)_{\< i_r \>} \}$ corresponds 
to $((\wF_\gs)_r, \wm_\gs)$, and $\gt \subset G(\wI)$ appears in the second $\sum$ of \eqref{cell diff} 
if and only if $(\wF_{\gt}, \wm_{\gt})=((\wF_\gs)_r, (\wm_\gs)_{\<i_r\>})$ 
for some $r \in B(\wF_\gs, \wm_\gs)$ by Proposition~\ref{gradient path}. 
Hence, if we forget ``coefficients", the differential map of $\wQ_\bullet$ and 
that of $\wP_\bullet$ are compatible with  the maps $e(\gs) \longmapsto e(\wF_\gs, \wm_\gs)$. 
 So it remains to check the equality of the coefficients.  

To define the coefficient $m(\cP)\in \ZZ$ used in \eqref{cell diff}, we 
fix an orientation of the simplex $X$. 
Identifying  $X$ with the power set $2^{G(\wI)}$, we set 
$[\gs:\gs'] = (-1)^r$ for $\gs = \{\wm_1,\dots,\wm_{q+1}\}$
and $\gs' = \gs \setminus \{\wm_r\}$
with $\wm_1 \sqsubset \wm_2 \sqsubset \cdots \sqsubset \wm_{q+1}$.
Then $m(\cP) \in \ZZ$ for a gradient path $\cP=\gs_0 \to \gs_1 \to  \cdots \to \gs_l$ is defined by 
$$m(\cP) = \prod_{\stackrel{1 \le i \le l}{(\gs_i \to \gs_{i-1}) \not \in A}}[\gs_{i-1} : \gs_i] \ \times 
\prod_{\stackrel{1 \le i \le l}{(\gs_i \to \gs_{i-1}) \in A}}(-[\gs_i: \gs_{i-1}]).$$

Let $\gs,\gt \subset G(\wI)$ be critical cells with $(\wF_{\gt}, \wm_{\gt})=((\wF_\gs)_r, (\wm_\gs)_{\<i_r\>})$ 
for some $r \in B(\wF_\gs, \wm_\gs)$, and $\cP: \gs \setminus \{ \, \wm_\gs \, \} \to \cdots \to 
\gt$ the gradient path constructed in the proof of 
Propositions~\ref{gradient path}. By Proposition~\ref{uniqueness}, 
$\cP$ is the unique one connecting $\gs \setminus \{ \, \wm_\gs \, \}$ with $\gt$. 
It suffices to show that $(-1)^q m(\cP) = (-1)^r$.  

The path $\cP$ is a succession of short  paths
\begin{align}
\gs^{(1)} \to \gs^{(2)} \to \gs^{(3)} \label{eq:unit_path}
\end{align}
such that $\# \gs^{(1)} = \# \gs^{(3)} = q$ and $\# \gs^{(2)} = q + 1$.
For convenience, we refer to a path of type~\eqref{eq:unit_path} 
as a {\em peak}.
According to the construction of $\cP$,
the peaks $\gs^{(1)} \to \gs^{(2)} \to \gs^{(3)}$ appearing in $\cP$
are classified into the following two types $\cA$ and $\cB$.
The type $\cA$ is such that $\gs^{(1)} \to \gs^{(2)}$ is given by insertion
an element into the $s$-th position of $\gs^{(1)}$ and $\gs^{(2)} \to \gs^{(3)}$
by deletion of the $(s+1)$-th element of $\gs^{(2)}$, for some $1 \le s \le r-1$.
To be more explicit, set $\gs^{(1)} := \{\wn_1,\dots,\wn_q\}$
with $\wn_1 \sqsubset \wn_2 \sqsubset \cdots \sqsubset \wn_q$. 
A peak of type $\cA$ is of the following form
\begin{align*}
\{\wn_1,\dots,\wn_q\} &\to \{\wn_1,\dots,\wn_{s-1},\wn_s',\wn_s,\cdots, \wn_q\} \\
&\to \{\wn_1,\dots,\wn_{s-1},\wn_s',\wn_{s+1},\cdots, \wn_q\},
\end{align*}
where $\wn_{s-1} \sqsubset \wn_s' \sqsubset \wn_s$.
The type $\cB$ is such that $\gs^{(1)} \to \gs^{(2)}$ is given by insertion
an element into the $r$-th position of $\gs^{(1)}$ and $\gs^{(2)} \to \gs^{(3)}$
by deletion of the $(q+1)$-th element of $\gs^{(2)}$. That is, 
\begin{align*}
\{\wn_1,\dots,\wn_q\} &\to \{\wn_1,\dots,\wn_{r-1},\wn_r',\wn_r,\cdots, \wn_q\} \\
&\to \{\wn_1,\dots,\wn_{r-1},\wn_r',\wn_r,\cdots, \wn_{q-1}\},
\end{align*}
where $\wn_{r-1} \sqsubset \wn_r' \sqsubset \wn_r$

Let $\cP' = (\gs^{(1)} \to \gs^{(2)} \to \gs^{(3)})$ be a peak of type $\cA$ or $\cB$. 
By the choice of orientation of $X$,
$$
m(\cP') = -[\gs^{(2)}:\gs^{(1)}][\gs^{(2)}:\gs^{(3)}] =\begin{cases}
1 & \text{if $ \cP' $ is of type $ \cA $} \\
(-1)^{q-r} & \text{if $ \cP' $ is of type $ \cB $.}
\end{cases}
$$
Hence it follows that $m(\cP) = (-1)^{l(q-r)}$, where $l$ is the number of the peaks of type $\cB$
which appears in $\cP$.

To count the number of the peaks of type $\cB$, we use the framework of the proof of Proposition~\ref{gradient path}.
It is clear that the paths constructed in Case 1 and Step 2 of Case 2 are successions of peaks of type $\cA$.
In Step 1 of Case 2, an easy observation shows that
the path is given by $(q-r)$-times repetition of construction of a path
$$
\cP_1\cdots \cP_{r-1} \cP_r,
$$
where $\cP_i$ is a peak of type $\cA$ for $i < r$, and $\cP_r$ is a peak of type $\cB$.
Summing up these observation, we have $l = q-r$ (Recall that $q = r$ in Case 1).
Therefore it follows that
$$
(-1)^qm(\cP) = (-1)^q \cdot (-1)^{(q-r)^2} = (-1)^{q + (q-r)} = (-1)^r,
$$
as desired.
\end{proof}

The following is easy. 

\begin{cor}
The free resolution $\wP_\bullet \otimes_{\wS} \wS/(\Theta)$ (resp.  $\wP_\bullet \otimes_{\wS} \wS/(\Theta_a)$) of 
$S/I$ (resp. $T/I^{\gamma(a)}$) is also a cellular resolution supported by $X_A$. 
In particular, these resolutions are Batzies-Welker type. 
\end{cor}

\section{Final Remarks}
For the formal definition of the {\it regularity} of a (finite) CW complex, consult suitable textbooks. 
For our purpose, a regular CW complex is  a CW complex such that for all $i \ge 0$ the closure $\overline{\gs}$  
of any $i$-cell $\gs$ is homeomorphic to an $i$-dimensional closed ball, 
and $\overline{\gs} \setminus \gs$ is the closure of the union of some $(i-1)$-cells.  
Recently, Mermin \cite{Mer} (see also \cite{Cl}) showed that the Eliahou-Kervaire resolution of a Borel fixed ideal 
(more generally, a stable monomial ideal) is cellular, and the supporting CW complex is regular. 
In the previous section, we showed that our resolution $\wP_\bullet$ is cellular. However,  
the regularity of the supporting complex $X_A$ constructed by discrete Morse theory is not clear, while we have the following. 

\begin{prop}
Let $X_A$ be the CW complex constructed in the previous section
(recall that there is a one-to-one correspondence between critical subsets $\gs \subset G(\wI)$ and 
cells $\gs_A$ of $X_A$). Then the following hold. 
\begin{itemize}
\item[(1)]  If the closure of an $(i+1)$-cell $\gs_A$ contains an $(i-1)$-cell 
$\gt_A$, there are exactly two cells between them. 
\item[(2)] The incidence number $[\gs_A : \gs'_A]$ is $1, -1$ or $0$ for all $\gs_A,\gs'_A$.  
\end{itemize}
\end{prop}

If $X_A$ is regular, the conditions of the above proposition hold obviously.

\begin{proof}
(1) Let $\gs'_A$ be an $i$-cell of $X_A$. 
As shown in the previous section, the closure of $\gs_A$ contains $\gs_A'$ if and only if 
$e(\wF_{\gs'}, \wm_{\gs'})$ appears in $\partial(e(\wF_{\gs}, \wm_{\gs}))$, where 
$\partial$ is the differential of $\wP_\bullet$.   
Hence the assertion follows from the analysis of $\partial$ given in \S3. 

(2) Since $[\gs_A:\gs'_A]$ coincides with the coefficient in $\partial$, the assertion is also clear. 
In other words, the assertion is a consequence of Proposition~\ref{uniqueness}. 
\end{proof}

Recall that if a Borel fixed ideal $I$ is generated in one degree then our resolution $\wP_\bullet$ of 
$\BoX(I)$ is equivalent to the resolution of Nagel and Reiner \cite{NR}. Their resolution is cellular, and the supporting 
CW complex is polytopal, hence is regular.  
Since their CW complex is contractible as shown in the proof of \cite[Theorem~3.13]{NR}, 
it can be taken as $X_A$.
Anyway, if $I$ is generated in one degree, the Nagel-Reiner resolution of $I$ is Batzies-Welker type, and 
the CW complex $X_A$ is regular.

\begin{exmp}\label{final remark}
Set $I :=(x_1^2, x_1x_2^2, x_1x_2x_3, x_1x_2x_4,x_1x_3^2,x_1x_3x_4)$. 
Then
$$
\wI = (x_{1,1}x_{1,2},x_{1,1}x_{2,2}x_{2,3}, x_{1,1}x_{2,2}x_{3,3}, x_{1,1}x_{2,2}x_{4,3}, x_{1,1}x_{3,2}x_{3,3}, x_{1,1}x_{3,2}x_{4,3}),
$$
and easy computation 
shows that the CW complex $X_A$, which supports our resolutions $\wP_\bullet$ of $\wS/\wI$ and  
$\wP_\bullet \otimes_{\wS} \wS/(\Theta)$ of $S/I$, is the one illustrated in Figure 3. 
The complex consists of a square pyramid and a tetrahedron glued along trigonal faces of each.  
For a Borel fixed ideal generated in one degree,  any face of the Nagel-Reiner CW complex is a product of several simplices 
(of positive dimension). 
Hence a square pyramid can not appear in the case of Nagel and Reiner.

\begin{figure}[htbp]
\begin{minipage}{.48\textwidth}
\begin{center}
\includegraphics[height=4cm, width=6.3cm]{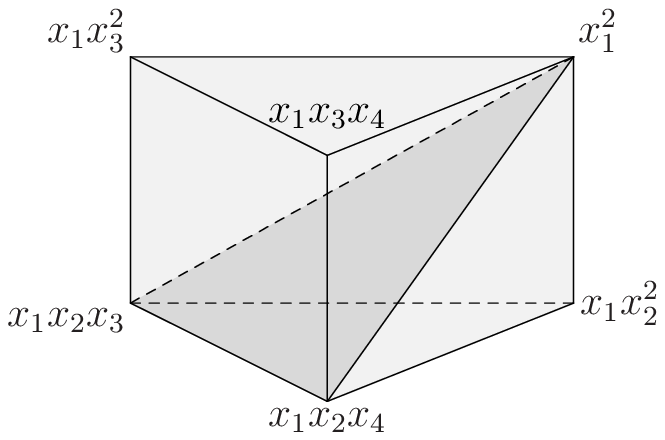}
\end{center}
\caption{}
\end{minipage}
\begin{minipage}{.48\textwidth}
\begin{center}
\includegraphics[height=4.2cm, width=6.5cm]{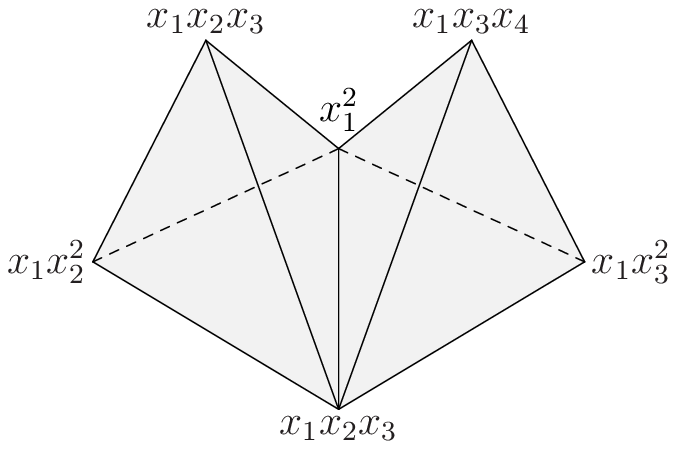}
\end{center}
\caption{}
\end{minipage}
\end{figure}

We remark that the Eliahou-Kervaire resolution of $I$ is supported by 
the CW complex illustrated in Figure 4. 
This complex consists of two tetrahedrons glued along edges of each. 
These figures show visually that the description of the  Eliahou-Kervaire resolution and 
that of ours are really different.
\end{exmp}

\begin{ques}
Is the CW complex $X_A$ is regular for a general Borel fixed ideal? 
\end{ques}

\section*{Acknowledgement}
We are grateful to Professor Volkmar Welker and the referees for valuable comments.

\end{document}